\documentclass[a4paper]{amsart}
\usepackage{graphicx}
\usepackage{amssymb}

\title[Warm-starting OA for parameterized convex MINLP]{Warm-starting outer approximation for parameterized convex MINLP}
\author[E. Tamm]{Erik Tamm}
\address{Department of Mathematics, KTH Royal Institute of Technology, Lindstedts\-v\"agen~25, Stockholm, SE-10044, Sweden}
\email{etamm@kth.se}
\urladdr{https://orcid.org/0009-0008-2878-1458}

\author[G. Eichfelder]{Gabriele Eichfelder}
\address{Institute of Mathematics, Technische Universit\"at Ilmenau, Weimarer Strasse~25, Ilmenau, DE-98693, Germany}
\email{gabriele.eichfelder@tu-ilmenau.de}
\urladdr{https://orcid.org/0000-0002-1938-6316}

\author[J. Kronqvist]{Jan Kronqvist}
\address{Department of Mathematics, KTH Royal Institute of Technology, Lindstedts\-v\"agen~25, Stockholm, SE-10044, Sweden}
\email{jankr@kth.se}
\urladdr{https://orcid.org/0000-0003-0299-5745}

\date{June 17, 2026}

\thanks{This work was supported by the Swedish Research Council (grant number 2022-03502).}

\usepackage{algorithm}
\usepackage{algpseudocode}

\usepackage{pgfplots}
\pgfplotsset{width=10cm,compat=1.18}

\pgfplotsset{
    /pgfplots/layers/Bowpark/.define layer set={
        axis background,axis grid,main,axis ticks,axis lines,axis tick labels,
        axis descriptions,axis foreground
    }{/pgfplots/layers/standard},
}
\usetikzlibrary{plotmarks}
\usetikzlibrary{pgfplots.groupplots}
\usepgfplotslibrary{groupplots}
\usetikzlibrary{matrix,positioning}

\usepackage{caption}
\usepackage{subcaption}
\usepackage{multirow} 

\usepackage{url} 

\newcommand{\norm}[2]{\| #1 \|_{#2}}
\newcommand{\R}{\mathbb{R}}
\newcommand{\Z}{\mathbb{Z}}
\newcommand{\F}{\mathcal{F}}
\newcommand{\X}{\mathcal{X}}
\newcommand{\T}{\mathcal{T}}
\newcommand{\FL}{\F_{L}}

\theoremstyle{plain}
\newtheorem{theorem}{Theorem}
\newtheorem{proposition}[theorem]{Proposition}
\newtheorem{lemma}[theorem]{Lemma}

\theoremstyle{definition}
\newtheorem{definition}[theorem]{Definition}
\newtheorem{example}[theorem]{Example}

\theoremstyle{remark}
\newtheorem{remark}[theorem]{Remark}

\begin{document}

\begin{abstract}
We address the challenge of efficiently solving parameterized sequences of convex Mixed-Integer Nonlinear Programming (MINLP) problems through warm-starting techniques. We focus on an outer approximation (OA) approach, for which we develop the theoretical foundation and present two warm-starting techniques for solving sequences of convex MINLPs. These types of problem sequences arise in several important applications, such as, multiobjective MINLPs using scalarization techniques, sparse linear regression, hybrid model predictive control, or simply in analyzing the impact of certain problem parameters.

The main contribution of this paper is the mathematical analysis of the proposed warm-starting framework for OA-based algorithms, which shows that a simple adaptation of the linear relaxation from one problem to the next can greatly improve computational performance. In the case that the parameters depend linearly on the parameter, we prove under some assumptions that one of the proposed warm-starting techniques results in only one OA iteration to find an optimal solution and verify optimality. Numerical results demonstrate noticeable performance improvements compared to two common initialization approaches, and show that the warm-starting can also in practice result in a single iteration to converge for several problems in the sequences. Our methods are especially effective for problems where consecutive problems in the sequence are similar, and where the integer part of the optimal solutions remains constant for several problems in the sequence. The results show that it is possible, both in theory and practice, to perform warm-starting to significantly enhance the computational efficiency of solving parameterized convex MINLPs.
\end{abstract}


\maketitle

\section{Introduction}\label{sec1}
Mixed-Integer Nonlinear Programming (MINLP) is a general class of problems that includes both integer and continuous variables, and nonlinear functions, enabling a wide range of applications to be modelled as such. For more details on MINLP applications, see e.g., \cite{belotti_2013, boukouvala_2016, floudas_1995, trespalacios_2014}. We specifically consider the case of sequences of convex MINLPs where the change between the problems in the sequence is known a priori and is small. We consider an MINLP to be convex if the objective function is convex and all nonlinear constraints are inequality constraints with convex functions. A more precise definition of the class of problems we consider is given in Subsection \ref{sec:background_prob_state}. Several algorithms for convex MINLPs have been proposed and investigated over the years, e.g., see \cite{ambrosio_2013, Kronqvist_2019} for an overview. However, to the best of the authors' knowledge there has been little work on warm-starting for sequences of problems. Existing work has mainly focused on warm-starting subproblems arising when solving a single convex MINLP, e.g., \cite{abhishek2010filmint, benson2011mixed}, and not on sequences of problems. The main research questions behind our work have been: ``Can the concept of warm-starting be successfully used when solving sequences of convex MINLP instances?'' and ``How can it be applied effectively?''. Due to the discrete nature of MINLP, it is not clear from a theoretical or practical point of view if warm-starting can be successfully used. How warm-starting should be done clearly depends on the algorithmic framework used for solving MINLP instances, and, therefore, we focus on one main algorithm together with the warm-starting approaches. We hope that this paper can be a starting point for more research into the area of sequences of MINLPs, both in terms of algorithms and applications.

We have restricted ourselves to convex MINLPs, since we can more effectively build upon the algorithms used to solve convex MINLPs to develop warm-starting techniques. We have chosen to focus on the Outer Approximation (OA) algorithm for solving MINLPs, and there are two main reasons behind this. First, the algorithm is commonly used and considered one of the most efficient ways for solving convex MINLPs. This can, for example, be seen in the review of software and algorithms to solve convex MINLPs \cite{Kronqvist_2019} and the algorithmic framework presented in \cite{Bonami_2008}. Secondly, the algorithm is simple to interpret and visualize, giving us a good intuition of the method and what information can be used to perform warm-starting. There are several variants of the OA algorithm, and we specifically chose the one presented in \cite{fletcher_1994} as the basis of this work.

In broad terms, any computational technique that takes advantage of prior information before solving the problem to reduce the computational burden can be considered a warm-starting method. Prior information could be knowledge gained from solving a similar problem, structural knowledge of the problem, or simply a good guess for the optimal solution. We focus on using, and adapting, the linear relaxation obtained from solving similar problems. Since solving the Mixed-Integer Linear Programming (MILP) subproblems typically is the computational bottleneck for the OA algorithm, the goal is to decrease the number of such subproblems. By starting from a good linear relaxation, we strive to reduce the number of iterations for the OA algorithm and the number of integer assignments that need to be explored to solve the problem.

The main contributions of this paper are I) the presentation of a framework for warm-starting OA-based algorithms, and II) showing that warm-starting the OA algorithm is possible for sequences of convex MINLPs both in theory and in practice. We present two simple techniques for warm-starting that build upon adjusting and reusing the previously obtained linear relaxation. In the case where the parameters depend linearly on the parameter, we prove under some assumptions that one of the proposed warm-starting techniques will successfully make the OA algorithm converge after only a single iteration. Numerical results show that terminating after a single iteration is not only a theoretical possibility, but also observed in practice. Overall the numerical results show great potential for the proposed warm-starting techniques for sequences where the change between problems is small, and we can greatly reduce the number of iterations and the time to solve the problems. The results also show that the trivial warm-starting technique of starting at a point near the optimal solution is not an effective warm-starting approach for a convex MINLP. Example \ref{ex:warm-starting} in Subsection \ref{sec:example_warm-starting} illustrates this, where starting at the optimal solution requires as many iterations of the OA algorithm to converge as the typical approach of initializing the algorithm by solving the continuous relaxation. This paper shows, through the theoretical and numerical results, that warm-starting applied in this setting can give significant performance improvements and that there is great potential for further research in new warm-starting techniques.

There has been a stream of research into parametric MINLP, which also focuses on MINLPs with some varying parameters. Parametric MINLP builds upon concepts from parametric MILP, and for more details on parametric MILP we refer to, e.g., \cite{dua2000algorithm,Gal_1995,mitsos2009parametric}. To the best of the authors' knowledge, the work on parametric MINLP has been focusing on the case where the integer variables, here denoted by $y$, exclusively appear linearly and on so-called right hand side parametrization. This means that the constraints are assumed to be of the form $g(x) + Ay \leq b + F\lambda$, where $\lambda$ is the parametrization parameter, e.g., see \cite{Acevedo_Salgueiro_2003,Dua_2009,Papalexandri_Dimkou_1998}. However, the main difference compared to our work is that in parametric MINLP the aim is to describe the optimal solution as a function of the parameter. Our focus is to efficiently solve the problem for different values of the parameters. We choose the term \textit{parameterized} MINLP to clarify the difference from the term \textit{parametric}, used for example in parametric programming. Furthermore, we do not enforce restrictions on the integer variables to appear linearly in the problem.

Multiobjective MINLP is a class of problems that initially sparked our interest in parameterized MINLPs. By solving a multiobjective problem using scalarization techniques, we obtain a parameterized MINLP with a highly structured change between problems. This leads to a sequence of MINLPs where each problem can be computationally challenging to solve. Therefore, warm-starting techniques are of great interest in the multiobjective problem setting.

A link between the efficient frontier of a multiobjective MILP and the value function of a related single-objective MILP was investigated in \cite{Fallah_2024}. The authors explored the structure of the value function by expanding results from the case of pure integer programs. They also showed an equivalence between the efficient frontier and the value function, thereby arguing that they are interchangeable in algorithmic development. Based on these findings, a cutting-plane algorithm to find the efficient frontier of a multiobjective MILP was developed. Similarly to parametric programming, the focus is on approximating the function in the image space. We reiterate that our goal is to find optimal solutions for different parameter values efficiently without claiming that all cases are covered, or that it is a suitable approximation of the objective value as a function of the parameter. Further developing the theory of the value function to the MINLP case could be an interesting direction for future research.

Other applications where warm-starting can be particularly useful include sparse linear regression and hybrid model predictive control. These applications, as well as multiobjective problems, are included as test problems in Section \ref{sec:num_exp} on numerical results.

The paper is organized as follows. First, in Section \ref{sec:background}, we present the parameterized convex MINLP framework and introduce some key concepts. Then, in Section \ref{sec:prob_form}, some theoretical results for the problem are formulated, and in Section \ref{sec:algorithm} the proposed algorithms are presented. In Section \ref{sec:num_exp} numerical experiments for five types of test problems are presented and the results are analyzed. Finally, we give some conclusions and ideas for future research in Section \ref{sec:conclusions}.

\section{Problem formulation and background}\label{sec:background}
In this section we give a brief overview of concepts and algorithms that we build the warm-starting methods on. First, in Subsection \ref{sec:notation}, some notation used is specified. In Subsection \ref{sec:background_prob_state} we then present the class of problems that we focus on. In Subsection \ref{sec:background_minlp} and \ref{sec:background_multiobj} some terminology used for MINLP and multiobjective optimization respectively is specified. The OA algorithm is described in Subsection \ref{sec:background_oa_static}. The algorithm is recalled in detail as we build the warm-starting methods upon it. Lastly, in Subsection \ref{sec:example_warm-starting}, an example is presented to show that starting in an optimal solution is not sufficient for successfully warm-starting the OA algorithm.

\subsection{Notation}\label{sec:notation}
We introduce some notation that is used in the paper. First, we define the index sets $[n] := \{1, 2, \ldots, n\}$ and $[n]_0 = \{0\} \cup [n]$ for any natural number $n$. For $x, y \in \R^n$ we say that $x \leq y$ if $x_i \leq y_i$ for all $i \in [n]$. By the notation $\| \cdot \|$ we denote the Euclidean norm unless otherwise specified. The open ball around a point $x \in \R^n$ with radius $\delta > 0$ is denoted as $$B(x, \delta) := \{u \in \R^n \mid \| x - u \| < \delta \}.$$

\subsection{Parameterized convex MINLP}\label{sec:background_prob_state}
Let $P \subset \R^k$ be a nonempty and connected set, i.e., there are no sets $A, B \subset \R^k$ such that $A \cap \text{cl}(B) = \text{cl}(A) \cap B = \varnothing$ and $A \cup B = P$ (see \cite[Definition 2.45]{Rudin_1976}). For parameters $\lambda \in P$, we consider the family of problems of the form
\begin{equation}\tag{$P_\lambda$}\label{eq:p_lambda}
    \begin{array}{rl}
        \displaystyle \min_{x, y} & c^\top x + d^\top y \\
        \text{s.t.} & g^\lambda(x,y) \leq 0, \\
         & h(x,y) \leq 0, \\
         & Ax+By \leq b, \\
         & x \in \R^n \text{, } y \in \Z^m
    \end{array}
\end{equation}
where $c \in \R^n$, $d \in \R^m$, $g^\lambda: \R^{n+m} \rightarrow \R^p$, $h: \R^{n+m} \rightarrow \R^q$, $A \in \R^{r \times n}$, $B \in \R^{r \times m}$ and $b \in \R^r$. We assume that each component in the functions $g^\lambda$ and $h$ is convex and once continuously differentiable. It is also assumed that the set $$ \FL := \{(x,y) \in \R^{n+m} \mid Ax + By \leq b \}$$ is a nonempty bounded set. As is commonly done, we assume that for each fixed $y$ such that the feasible set is nonempty the linear independence constraint qualification (LICQ) holds in all optimal solutions (see e.g. \cite{fletcher_1994}). Note that $g^\lambda$ depends on the parameter $\lambda \in P \subset \R^k$. We finally assume that for all $(x, y) \in \FL$, the function $\lambda \mapsto g^\lambda(x, y)$ is continuous in $\lambda$ in $P$. The intuition behind the continuity assumption is that the feasible sets of ($P_{\lambda^0}$) and ($P_{\lambda^1}$) will be ``similar'' if $\| \lambda^0 - \lambda^1 \|$ is small enough. Without this assumption the feasible sets could change drastically, meaning that efficient warm-starting will be difficult or impossible to achieve. We are focusing on warm-starting in the setting where the feasible sets of consecutive problems are similar.

We refer to a problem of the form ($P_\lambda$) as a parameterized convex MINLP, which depends on a parameter $\lambda$. The goal is to solve the problem for a sequence of values of $\lambda \in P$.

We can assume without loss of generality that the objective function is linear since if it has a nonlinear objective function it is possible to use an epigraphical reformulation to obtain a linear objective function.

Convex MINLP can be solved using several techniques, e.g., branch and bound \cite{Dakin_1965, Land_Doig_1960, Gupta_1986}, generalized Benders decomposition \cite{Geoffrion_1972}, extended cutting planes \cite{Westerlund_1995} or outer approximation \cite{Duran_Grossmann_1986, fletcher_1994}. For an overview of the methods see for example \cite{bonami_2012, ambrosio_2013, Kronqvist_2019}. Even though convex MINLP is simpler to solve than general MINLP, it is still not trivial. The convexity of the problem means that the algorithms are simpler than for the more general non-convex setting. Therefore, we are able to progress further with the theoretical analysis and algorithmic design. We are optimistic that concepts and ideas from this paper can be extended to the non-convex setting, but we limit ourselves to the convex setting.

\subsection{Basic concepts of outer approximation}\label{sec:background_minlp}
Key components of the OA algorithm are the concepts of a valid linear constraint and a linear relaxation. A linear inequality $a^\top x \leq b$ is called valid for a set $\mathcal{S} \subset \R^n$ if for every point $\bar{x} \in \mathcal{S}$ it holds that $a^\top \bar{x} \leq b$. A set $T := \{x \in \R^n \mid (a^j)^\top x \leq b_j \text{ for } j \in [m] \}$ defined by $m$ linear inequalities forms a linear relaxation of a set $\mathcal{S} \subset \R^n$ if $\mathcal{S}$ is a subset of $T$. For more details, see, e.g., \cite{Conforti_2014}.

A well-known property of convex and continuously differentiable functions $f\colon\R^n \rightarrow \R$ is that
    \begin{equation*}
        f(y) + \nabla f(y)^\top (x - y) \leq f(x) \mbox{ for all }x,y\in\R^n,
    \end{equation*}
see, e.g., \cite[pp. 69--70]{Boyd_2004}. It follows that, if $\mathcal{S} = \{x \in \R^n \mid f(x) \leq 0\}$ for some convex and continuously differentiable function $f\colon\R^n \rightarrow \R$ and $\hat{x} \in \R^n$, the linear inequality $f(\hat{x}) + \nabla f(\hat{x})^\top (x - \hat{x}) \leq 0$ is valid for $\mathcal{S}$. Furthermore, if $\hat{x} \notin \mathcal{S}$, the inequality separates $\hat{x}$ from $\mathcal{S}$. This linear inequality is called a gradient cut. The basic idea of the OA algorithm is to iteratively add gradient cuts until a good enough approximation of the nonlinear constraints is obtained.

\subsection{Basic concepts of multiobjective optimization}\label{sec:background_multiobj}
It was mentioned in the introduction that scalarization techniques for multiobjective optimization have been an important special case that initially motivated this work. This section will define some basic concepts from multiobjective optimization. For a general introduction to multiobjective optimization, we refer, for instance, to \cite{Ehrgott_2005, Eichfelder_2021, Miettinen_1999}.

A multiobjective optimization problem is an optimization problem where the goal is to minimize several, typically competing, objective functions simultaneously. The problem is stated as
\begin{equation}\label{MOP}\tag{MOP}
    \begin{array}{rl}
        \displaystyle \min_x & f(x) = (f_1(x), f_2(x), \ldots, f_p(x)) \\
        \text{s.t.} & x \in \mathcal{S}
    \end{array}
\end{equation}
where each $f_j: \R^n \rightarrow \R$ and $\mathcal{S} \subset \R^n$ is the feasible set. In theory, there can exist a feasible point where each $f_j$ is minimized, but in practice this generally does not happen. The goal is instead to find either efficient points (corresponding to the minimal solutions in the pre-image space) or nondominated points (corresponding to the optimal values in the image space).

Recall that $\bar{x} \in \mathcal{S}$ is efficient for \eqref{MOP} if there is no $x \in \mathcal{S}$ with $f_i(x) \leq f_i(\bar{x})$ for all $i \in [p]$ and with $f_j(x) < f_j(\bar{x})$ for at least one $j \in [p]$. In other words, it is not possible to improve the value of one of the objective functions without worsening the others. The corresponding point $f(\bar{x})$ is called nondominated. There is also a weaker notion of optimality called weakly efficient. A point $\bar{x} \in \mathcal{S}$ is called weakly efficient for \eqref{MOP} if there is no $x \in \mathcal{S}$ with $f_j(x) < f_j(\bar{x})$ for all $j \in [p]$, i.e., which has a strictly lower value for all objective functions. The corresponding point $f(\bar{x})$ is called weakly nondominated.

To solve a multiobjective optimization problem we can use so-called nonlinear scalarization techniques which can find all nondominated points even in the case of non-convex problems meaning that it is suitable for integer optimization. The most well-known of these nonlinear scalarization techniques is the so-called $\varepsilon$-constraint method. The method considers the problem of minimizing one of the objective functions, for instance $f_p$, while constraining the other elements of $f$ by upper bounds $\varepsilon_j$. For consistency of notation, the upper bounds will be denoted by $\lambda_j$ since they are the parameters of this scalarization. We obtain the problem as
\begin{equation}\label{eq:eps_prob}
\tag{$\text{P}_{\text{MOP}}(\lambda)$}
    \begin{array}{rl}
        \displaystyle \min_x & f_p(x) \\
        \text{s.t.} & f_j(x) \leq \lambda_j \text{, } j \in [p-1], \\
        & x \in \mathcal{S}.
    \end{array}
\end{equation}
 
Any optimal solution to \eqref{eq:eps_prob} is at least a weakly efficient point for \eqref{MOP}. In case the optimal solution of \eqref{eq:eps_prob} is unique, then it is even an efficient point for \eqref{MOP}. On the other hand, any efficient point for \eqref{MOP}, or correspondingly a nondominated point for \eqref{MOP}, can be found with the $\varepsilon$-constraint method by an appropriate parameter choice.   

Choosing the values of the $\lambda_j$ is not trivial if $p \geq 3$. To finely discretize the values for each $\lambda_j$  within a range and to solve the whole family of problems \eqref{eq:eps_prob} for all these parameters can make the problem too computationally demanding. For $p = 2$, referred to as biobjective, the problem \eqref{eq:eps_prob} reduces to minimizing one objective function and constraining the other. This means that we only need to solve \eqref{eq:eps_prob} for a scalar sequence of values for $\lambda$ to find at least a subset of the set of nondominated points.

\subsection{Outer Approximation Algorithm}\label{sec:background_oa_static}
The OA algorithm was first presented by Duran and Grossmann \cite{Duran_Grossmann_1986}. In their work, they follow the setting for generalized Benders decomposition by assuming that the integer variables are linearly present in the constraints. The version considered in this paper was presented by Leyffer and Fletcher \cite{fletcher_1994}, who generalized the idea to allow integer variables in nonlinear expressions and considered a different approach to handling infeasible integer assignments. Since the main concepts of the OA algorithm are important for understanding the method for the parameterized problem, we will here describe the algorithm for solving a convex MINLP without parameterized constraints.

The OA algorithm solves convex MINLPs, i.e., we consider problems of the form
\begin{equation}\tag{$P$}\label{eq:minlp_static}
    \begin{array}{rl}
        \displaystyle \min_{x, y} & c^\top x + d^\top y \\
        \text{s.t.} & g(x, y) \leq 0,\\
         & Ax + By \leq b, \\
         & x \in \R^n \text{, } y \in \Z^m
    \end{array}
\end{equation}
where $g: \R^{n+m} \rightarrow \R^p$, $A \in \R^{r \times n}$, $B \in \R^{r \times m}$ and $b \in \R^r$. We assume that each component $g_j$ of the function $g$ is convex and once continuously differentiable and that the set $$\FL := \{(x, y) \in \R^{n+m} \mid Ax + By \leq b\}$$ is nonempty and bounded. Finally, it is assumed that for any fixed $\hat{y}$ such that the feasible set of \eqref{eq:minlp_static} is nonempty, the LICQ holds at all optimal solutions of the reduced continuous problem, as is commonly done, e.g., in \cite{fletcher_1994}. This assumption is related to the so-called NLP subproblem which is presented below.

We use $\F$ to refer to the feasible set of \eqref{eq:minlp_static}. We can now define the three subproblems used in the algorithm. First, for a fixed value of $y = \hat{y}\in\Z^m$ we define the problem
\begin{equation*}\tag{NLP($\hat{y}$)}\label{eq:nlp_static}
    \begin{array}{rl}
        \displaystyle \min_x & c^\top x + d^\top \hat{y} \\
        \text{s.t.} & (x, \hat{y}) \in \F.
    \end{array}
\end{equation*}
The problem \eqref{eq:nlp_static} is a nonlinear continuous convex problem and a restriction of \eqref{eq:minlp_static}. This means that, if feasible, the minimal value is an upper bound for the minimal value of \eqref{eq:minlp_static}. As stated above, it is assumed that for any $\hat{y}$ such that \eqref{eq:nlp_static} is feasible, the LICQ holds at all optimal solutions.

It is however possible that \eqref{eq:nlp_static} is infeasible. For this case we define the problem
\begin{equation}\tag{NLP-Feas($\hat{y}$)}\label{eq_prob_feas}
    \begin{array}{rl}
        \displaystyle \min_{r, x} & r \\
        \text{s.t.} & g_i(x, \hat{y}) \leq r \text{ for } i \in [p],\\
        & Ax+B\hat{y} \leq b, \\
        & x \in \R^n \text{, } r \in \R.
    \end{array}
\end{equation}
The feasibility problem finds the point $x$ for which $(x, \hat{y})$ minimizes the maximal constraint violation.

Lastly, we define the third subproblem for which the feasible set is a linear relaxation of $\F$, meaning that any feasible point for \eqref{eq:minlp_static} will be feasible in the subproblem, thus forming a relaxation. Let $$\X = \{(x^1, y^1), (x^2, y^2), \ldots, (x^l, y^l)\} \subset \R^n \times \Z^m$$ be a finite set of points. We refer to $\X$ as the set of linearization points. We refer to the linearization of $g_j$ at a point $(\bar{x}, \bar{y})$ as $$\tilde{g}_j^{(\bar{x}, \bar{y})}(x,y) := g_j(\bar{x}, \bar{y}) + \nabla g_j(\bar{x}, \bar{y})^\top \begin{pmatrix} x - \bar{x} \\ y - \bar{y} \end{pmatrix}.$$ Now define $\Omega(\X)$ as
\begin{equation*}
    \Omega(\X) := \Bigg\{(x,y) \in \R^n \times \Z^m \Bigg\vert 
    \begin{array}{rl}
        \tilde{g}_j^{(x^i, y^i)}(x,y) \leq 0, & \forall (x^i, y^i) \in \X, \forall j \in [k]\\
        Ax + By \leq b & \\
    \end{array}
    \Bigg\}.
\end{equation*}

Since all functions $g_j$ are convex, $\Omega(\X)$ defines a linear relaxation of $\F$. Finally, we define the problem
\begin{equation}\tag{MILP($\X$)}\label{eq:milp_static}
    \begin{array}{rl}
        \displaystyle \min_{x, y} & c^\top x + d^\top y \\
        \text{s.t.} & (x, y) \in \Omega(\X). \\
    \end{array}
\end{equation}
This problem is a relaxation of \eqref{eq:minlp_static} meaning that the minimal value is a lower bound for the minimal value of \eqref{eq:minlp_static}. The subproblem \eqref{eq:milp_static} is a simpler problem compared to \eqref{eq:minlp_static} as it is a MILP.

Using these subproblems we can define the OA algorithm from \cite{fletcher_1994} in Algorithm \ref{algo:OA}.
\begin{algorithm}[t]
    \caption{Outer Approximation}\label{algo:OA}
    \begin{algorithmic}
        \Require $y^0 \in \Z^m$, $\X^0 \subset \R^n \times \Z^m$ and $\varepsilon \geq 0$ are given. If $\X^0$ is not given, set $\X^0 = \varnothing$
        \State \textbf{Set: } $LB = -\infty$, $UB = \infty$ and $k = 0$.
        \Repeat
        \If{(NLP($y^k$)) is feasible}
            \State Solve (NLP($y^k$)) $\rightarrow$ optimal solution $x^k$ and optimal value $f^k$.
            \If{$f^k < UB$}
                \State $x^* \gets x^k$, $y^* \gets y^k$ and $UB \gets f^k$.
            \EndIf
        \Else
            \State Solve (NLP-Feas($y^k$)) $\rightarrow$ optimal solution $x^k$.
        \EndIf
        \State $\X^{k+1} \gets \X^{k} \cup \{(x^k, y^k)\}$.
        \State $k \gets k + 1$.
        \State Solve (MILP($\X^{k}$)) $\rightarrow$ optimal solution $(\hat{x}, y^k)$ and optimal value $LB$.
        \Until{$UB - LB \leq \varepsilon$}
        \State \Return Optimal solution ($x^*$, $y^*$) and the linearization points $\X^k$
    \end{algorithmic}
\end{algorithm}
The algorithm terminates when the optimality gap, i.e., the difference between the upper and the lower bound of the optimal value, is small enough to conclude optimality. In theory, our assumptions yield that it is guaranteed that the algorithm will, for $\varepsilon = 0$, terminate with $UB = LB$ after a finite number of iterations (see Theorem 2 in \cite{fletcher_1994}). The original algorithm in \cite{fletcher_1994} adds an upper bound on the objective function in the MILP subproblem and optimality is concluded when the MILP subproblem is infeasible. This is, however, not needed in this setting where all functions are once continuously differentiable.

\subsection{Illustrative example for warm-starting}\label{sec:example_warm-starting}
The most common way of doing warm-starting is to initialize the algorithm at a point that is close to the optimal solution for the last parameter value. In the setting of a parameterized problem, we consider that the parameterized problem has been solved for some specific parameter value. If the parameter is then slightly changed, it is possible to use the optimal solution found for the original parameter value and use that as the starting point for the new parameter value. Other nearby points would also be considered good starting points.

The idea of using an optimal solution or a near optimal point for warm-starting works well in many settings, most noticeably in the simplex method and in continuous optimization. However, for convex MINLP this simple approach is not sufficient. The following example illustrates why using the optimal solution might not lead to a good warm-starting approach for the OA algorithm. The example shows that even when starting at the optimal solution, several iterations are needed to conclude optimality. Furthermore, there is no gain compared to a common way of initializing the OA algorithm using the continuous relaxation.

\begin{example}\label{ex:warm-starting}
Consider the problem
\begin{equation}\label{eq:worst-case-warm-starting}
    \begin{array}{rl}
        \displaystyle \min_{x, y} & -2x-y \\
        \text{s.t.} & 3x^2+2y^2-2xy+3x-4y \leq 5.3, \\
        & -10x + y \leq 4, \\
        & x, y \in [-2, 10], \\
        & x \in \R, y \in \Z.
    \end{array}
\end{equation}

\begin{figure}[ht]
    \centering
    \begin{tikzpicture}
        \begin{axis}[
        width = 0.8\textwidth,
        xlabel = $x$,
        ylabel = $y$,
        xmin = -4.5,
        xmax = 4.5,
        ymin = -2.5,
        ymax = 9.5,
        ytick = {-2, -1, 0, 1, 2, 3, 4, 5, 6, 7, 8, 9},
        xtick = {-4, -2, 0, 2, 4},
        axis equal image,
        ylabel style={rotate=-90},
        ymajorgrids=true,
        grid style=dashed,
        set layers = Bowpark,
        colormap = {bw}{color = (black) color = (black)}
        ]

        \draw[black] (axis cs:-0.2,0.9) ellipse[
            rotate = 58.28,
            x radius = 2.31402,
            y radius = 1.43014
            ];
    
        \addplot[domain = (-2:2),
                samples = 50,
                color = black
                ]
                {4 + 10*x};
        
        \addplot[domain = -2.5:2.5,
                samples = 50,
                color = blue
                ]
                {(1/0.49)*(7.04-4.02*x)};

        \addplot[mark = *,
                point meta = explicit symbolic,
                nodes near coords,
                color = blue]
                coordinates {(1.51, 2)} node[right] {$(x^*, y^*)$};

        \addplot[mark = *,
                point meta = explicit symbolic,
                color = blue,
                nodes near coords]
                coordinates {(0.65, 9)} node[below right] {$(\hat{x}, \hat{y})$};

        \addplot[mark = *,
                point meta = explicit symbolic,
                nodes near coords,
                color = red]
                coordinates{(1.37, 2.47)} node[above right] {$(\bar{x}, \bar{y})$};

        \addplot[domain = -2.5:4.5,
                samples = 10,
                color = red]
                {(16.37 - 6.28*x)/3.14};

        \addplot[domain = 0.370901:0.629099,
                samples = 50,
                color = green]
                {3};
        \addplot[domain = -0.2:1.50624,
                samples = 50,
                color = green]
                {2};
        \addplot[domain = -0.3:1.40213,
                samples = 50,
                color = green]
                {1};
        \addplot[domain = -0.4:0.92009,
                samples = 50,
                color = green]
                {0};
        \addplot[domain = -0.5:-0.15428,
                samples = 50,
                color = green]
                {-1};  
        \end{axis}
    \end{tikzpicture}

    \caption{Illustration of Example \ref{ex:warm-starting}. The black line and the black ellipse correspond to the constraints. The green lines correspond to the feasible set. The blue line is the linearization at the optimal solution. The red line is the linearization at the optimal solution of the continuous relaxation of \eqref{eq:worst-case-warm-starting}.}
    \label{fig:worst-case-warm-starting}
\end{figure}

Figure \ref{fig:worst-case-warm-starting} illustrates the optimization problem \eqref{eq:worst-case-warm-starting}. Since $y$ is restricted to $\Z$, horizontal grey dashed lines are included to illustrate the integrality of $y$. The line and the ellipse in black correspond to the two constraints, meaning that the feasible set of the problem corresponds to the green lines.

The optimal solution of \eqref{eq:worst-case-warm-starting} is $x^* \approx 1.5$, $y^* = 2$. If the OA algorithm starts with $y^0 = y^* = 2$, the problem (NLP($\hat{y} = 2$)) is solved with an optimal solution $x^*$. Then $\X^0 = \{(x^*, y^*)\}$, meaning that the nonlinear constraint is linearized at $(x^*, y^*)$. This produces the cut shown by the blue line. The feasible set of (MILP($\X^0$)) is now all grey dashed lines inside the triangle created by the black and blue lines. If we solve (MILP($\X^0$)), the optimal solution is $\hat{x} \approx 0.65, \hat{y} =9$. This point is far from the optimal solution. In fact, altogether four iterations are needed to prove that $(x^*, y^*)$ is indeed optimal. Similar observations were made by Fletcher and Leyffer \cite{fletcher_1994}. They constructed a worst-case example which starts at the integer point adjacent to the optimal solution, but still has to visit all feasible integer points to converge.

We now compare this to a common way of initializing the OA algorithm. This method is to solve the continuous relaxation and use the optimal solution to obtain an initial cut. The next step is to solve the MILP subproblem to obtain a point that respects the integrality constraint. Applying this to our example, we start by solving the continuous relaxation, that is solving \eqref{eq:worst-case-warm-starting} but letting $y$ belong to $\R$, to obtain $\bar{x} \approx 1.37, \bar{y} \approx 2.47$. Even though $\bar{y} \notin \Z$, we can linearize the nonlinear constraint in $(\bar{x}, \bar{y})$ to get the red line in Figure \ref{fig:worst-case-warm-starting}. The added cut is now parallel to the level set of the objective function meaning that any integer feasible point on the line is optimal to the MILP subproblem. The solution process of the outer approximation will therefore depend on which optimal solution is found for the MILP subproblem. In the best case scenario, only the integer assignments $y = 2$ and $y = 3$ are needed to converge, corresponding to two iterations of the OA algorithm, while in the worst case scenario, 5 iterations are needed for convergence.
\end{example}

The conclusion of Example \ref{ex:warm-starting} is that in this case, warm-starting with an optimal solution has no benefit over a common way of initializing the OA algorithm which uses no prior knowledge of the problem. This example illustrates the need for more advanced warm-starting techniques to improve performance, since more information than the optimal solution is needed to obtain stable warm-starting techniques. This observation is also supported by the numerical experiments conducted.

\section{Theoretical justifications}\label{sec:prob_form}
Remember that we are interested in problems of the form \eqref{eq:p_lambda} where some of the constraints are parameterized by $\lambda$. For any fixed value of $\lambda$, this problem is of the same form as \eqref{eq:minlp_static} with the only difference that the nonlinear constraints are split into two functions.

By using the same definitions of the subproblems as in Section \ref{sec:background_oa_static} with the only addition that we include the value of the parameter, we get the following definitions.
\begin{definition}
    The notation $\F_\lambda$, (NLP($\hat{y}$, $\lambda$), (NLP-Feas($\hat{y}$, $\lambda$)), $\Omega(\X, \lambda)$ and (MILP$(\X, \lambda)$) refers to the defined sets and subproblems in Section \ref{sec:background_oa_static} for some specific value of $\lambda$.
\end{definition}

Using these definitions, we will develop some properties that are useful for understanding the OA algorithm in this setting and we use these properties to perform warm-starting. Since the main idea of our warm-starting framework is to reuse the set $\Omega(\X, \lambda)$ for the next $\lambda$ in the sequence we will first develop some properties of this set. After that we will consider how these properties affect the convergence of the OA algorithm since we hope to reduce the number of iterations needed to converge.

\subsection{Properties}\label{sec:theory-properties}
We first present an important property of the outer approximation framework that follows directly from convexity. Independent of the choice of  points in $\X$ and $\lambda$, the set defined by $\Omega(\X, \lambda)$ will be a linear relaxation of $\F_{\lambda}$. The interpretation of this fact is that there is nothing inherent in the way the OA algorithm constructs $\X$ that makes the cuts valid. We formalize this in Proposition \ref{prop:valid_oa}.
\begin{proposition}\label{prop:valid_oa}
    For any set of points $\X \subset \R^n \times \Z^m$ and $\lambda \in P \subset \R^k$ it holds that $\Omega(\X, \lambda)$ is a linear relaxation of $\F_\lambda$.
\end{proposition}
\begin{proof}
    This follows from the convexity of $g^\lambda$ and $h$.
\end{proof}

For the forthcoming Theorem \ref{thm:stability}, we use results involving the tangent cone, the set of linearized feasible directions and necessary first order optimality conditions expressed using the tangent cone. For completeness, the definitions are included here.

We first express a necessary first order optimality conditions using the tangent cone. This adapts Theorem 12.3 in \cite[p. 325]{Nocedal_Wright_2006} for (NLP($\hat{y}$, $\lambda$)).
\begin{proposition}\label{prop:1st_order_opt}
    If $x^*$ is a locally optimal solution of (NLP($\hat{y}$, $\lambda$)), then
    \begin{equation*}
        c^\top d \geq 0 \text{, for all } d \in \T_{\hat{y}, \lambda}(x^*)   
    \end{equation*}
    where $\T_{\hat{y}, \lambda}(x^*)$ is the tangent cone of the feasible set of (NLP($\hat{y}$, $\lambda$)) in $x^*$.
\end{proposition}

Now consider the definition of the cone of linearized feasible directions adapted from \cite[p. 316]{Nocedal_Wright_2006}.
\begin{definition}\label{def:linearized_set}
    Define first the following index sets:
    \begin{align*}
        I^\lambda(x, y) & := \{j \in [p] \mid g^\lambda_j(x, y) = 0 \} \\
        J(x, y) & := \{j \in [q] \mid h_j(x, y) = 0 \} \\
        K(x, y) & := \{j \in [r] \mid (a^j)^\top x + (b^j)^\top y = 0 \}
    \end{align*}
    where $a^j$ and $b^j$ are the $j$th rows of $A$ and $B$ respectively expressed as column vectors.
    
    The cone of linearized feasible directions $d$ at the feasible point $x$ for (NLP($\hat{y}, \lambda$)) is
    \begin{equation*}
        \mathcal{L}_{\hat{y}, \lambda}(x) = \Bigg\{ d\in\R^{n}\ \Bigg|\ 
        \begin{array}{rl}
            d^\top \nabla_x g^\lambda_j(x, \hat{y}) \leq 0 & \text{for all } j \in I^\lambda(x, \hat{y}) \\
            d^\top \nabla_x h_j(x, \hat{y}) \leq 0 & \text{for all } j \in J(x, \hat{y}) \\
            d^\top a^j \leq 0 & \text{for all } j \in K(x, \hat{y})
        \end{array}
        \Bigg\}
    \end{equation*} 
\end{definition}

\begin{remark}\label{remark:licq}
    Since we assume that the LICQ holds for (NLP($\hat{y}$, $\lambda$)) for any optimal solution $x$, we know that $\T_{\hat{y}, \lambda}(x) = \mathcal{L}_{\hat{y}, \lambda}(x)$ (see \cite[p. 323]{Nocedal_Wright_2006}).
\end{remark}

In the framework of warm-starting for parameterized MINLP that we use, reusing the set $\Omega(\X, \lambda)$ is a key component. We therefore want to understand what the relation is between the optimal solutions of \eqref{eq:p_lambda} and (MILP($\X$, $\lambda$)) when $\lambda$ changes. Specifically, consider the situation where ($P_{\lambda^0}$) has been solved with $\Omega(\X, \lambda^0)$ as the linear relaxation. If we now pick $\lambda^1 \in B(\lambda^0, \delta) \cap P$ for some small $\delta > 0$, will the solution of (MILP($\X$, $\lambda^1$)) directly give the solution for ($P_{\lambda^1}$) meaning that only one iteration of the OA algorithm is needed? This is not true in general as we show with Example \ref{ex:int_jump}. However, in some situations, depending on the optimal solution of (MILP($\X, \lambda^1$)), we do obtain the optimal solution of ($P_{\lambda^1}$) in one iteration, which is shown in Theorem \ref{thm:stability}. Furthermore, a sufficient condition for when this happens is given in Theorem \ref{thm:lambda_stable}.

\begin{example}\label{ex:int_jump}
    Consider, for $\lambda \geq 0$, the problem
    \begin{equation}\label{eq:int_jump}
        \begin{array}{rl}
            \displaystyle \min_{x, y} & -x \\
            \text{s.t.} & x^2 + (y-1)^2 \leq 1 + \lambda, \\
            & x^2/4 + 4y^2 \leq 1, \\
            & 3x + y \leq 3, \\
            & x, y \in [-2, 2], \\
            & x \in \R, y \in \Z.
        \end{array}
    \end{equation}

    \begin{figure}[ht]
        \centering
            \begin{subfigure}[b]{\textwidth}
                \centering
                \begin{tikzpicture}
                    \begin{axis}[
                    xmin = -2,
                    xmax = 2,
                    ymin = -1,
                    ymax = 1,
                    xtick = {-2, -1, 0, 1, 2},
                    ytick = {-1, 0, 1},
                    ymajorgrids=true,
                    grid style=dashed,
                    axis equal image,
                    width = 0.8\textwidth,
                    set layers = Bowpark,
                    colormap = {bb}{color = (black) color = (black)},
                    ]

                    \addplot[domain = (2/3):(4/3),
                            samples = 10,
                            color = black
                    ]
                    {3-3*x};

                    \addplot[domain = -1:2,
                            samples = 20,
                            color = blue]
                    {(2/3)*x - (4/9)};

                    \draw[black] (axis cs:0,1) circle[radius = sqrt(1+(4/9))];

                    \draw[black] (axis cs:0,0) ellipse[
                    x radius = 2,
                    y radius = 0.5
                    ];
                
                    \addplot[mark = *,
                            point meta = explicit symbolic,
                            nodes near coords]
                            coordinates {(2/3, 0)} node[above left] {$(2/3, 0)$};

                    \addplot[domain = -2/3:2/3,
                            samples = 10,
                            color = green]
                            {0};

                    \addplot[mark = *,
                            point meta = explicit symbolic,
                            nodes near coords]
                            coordinates {(2/3, 1)} node[below left] {$(2/3, 1)$};
                
                    \end{axis}
                \end{tikzpicture}
                \caption{Illustration of \eqref{eq:int_jump} with $\lambda^0 = 4/9$ and cut added in $(2/3, 0)$.}
                \label{fig:int_jump_t0}
            \end{subfigure}
            
            \begin{subfigure}[b]{\textwidth}
                \centering
                \begin{tikzpicture}
                    \begin{axis}[
                    xmin = -2,
                    xmax = 2,
                    ymin = -1,
                    ymax = 1,
                    xtick = {-2, -1, 0, 1, 2},
                    ytick = {-1, 0, 1},
                    ymajorgrids=true,
                    grid style=dashed,
                    axis equal image,
                    width = 0.8\textwidth,
                    set layers = Bowpark,
                    colormap = {bb}{color = (black) color = (black)}
                    ]

                    \addplot[domain = (2/3):(4/3),
                            samples = 10,
                            color = black,
                    ]
                    {3-3*x};

                    \addplot[domain = -2:2,
                            samples = 10,
                            color = blue]
                            {0.32*x-0.1};

                    \addplot[domain = -2:2,
                            samples = 20,
                            color = blue]
                            {(2/3)*x - (2/9) - 0.05};

                    \draw[black] (axis cs:0,1) circle[radius = sqrt(1.1)];

                    \draw[black] (axis cs:0,0) ellipse[
                    x radius = 2,
                    y radius = 0.5
                    ];
                
                    \addplot[mark = *,
                            point meta = explicit symbolic,
                            nodes near coords]
                            coordinates {(2/3, 0)} node[below] {$(2/3, 0)$};

                    \addplot[mark = *,
                            point meta = explicit symbolic,
                            nodes near coords]
                            coordinates {(sqrt(0.1), 0)} node[above left] {$(\sqrt{0.1}, 0)$};

                    \addplot[mark = *,
                            point meta = explicit symbolic,
                            nodes near coords]
                            coordinates {(2/3, 1)} node[below left] {$(2/3, 1)$};

                    \addplot[domain = -0.316:0.316,
                            samples = 10,
                            color = green]
                            {0};
                    \end{axis}
                \end{tikzpicture}
                \caption{Illustration of \eqref{eq:int_jump} with $\lambda^1 = 0.1$ and cut added in $(2/3, 0)$ and $(\sqrt{0.1}, 0)$.}
                \label{fig:int_jump_t1}
            \end{subfigure}
        \caption{Illustration of Example \ref{ex:int_jump}. The black lines and curves correspond to the constraints. The green lines correspond to the feasible set. The blue lines correspond to the linearizations.}
        \label{fig:int_jump}
    \end{figure}

    Problem \eqref{eq:int_jump} is plotted in Figure \ref{fig:int_jump} where \ref{fig:int_jump_t0} and \ref{fig:int_jump_t1} correspond to $\lambda^0 = 4/9$ and $\lambda^1 = 0.1$ respectively. Note that $y = 0$ is the only feasible integer assignment plotted as the green line.
    
    Suppose that we want to solve the problem first for $\lambda^0 = 4/9$ and then $\lambda^1 = 0.1$. Furthermore, assume that we start the OA algorithm with $\hat{y} = 0$. The optimal solution of (NLP($\hat{y} = 0, \lambda^0 = 4/9$)) is then $x = 2/3$. Therefore $\X^0 = \{(2/3,0)\}$, meaning that $(x, y) = (2/3, 0)$ is the first linearization point. The linearization is plotted as the blue line in Figure \ref{fig:int_jump_t0}. There is only one linearization in the plot since the linearization obtained from the ellipse is $x \leq 30/9$. We also get that $UB = -2/3$ since the minimal value of the NLP subproblem is an upper bound of the minimal value of the MINLP. The solution of (MILP($\X^0, \lambda^0 = 4/9$)) is then $(2/3,0)$ meaning that $LB = -2/3$ since the minimal value of the MILP subproblem is a lower bound of the minimal value of the MINLP. Hence $LB = UB$ and the OA algorithm has converged with the optimal solution $(x, y) = (2/3, 0)$.
    
    We now want to solve \eqref{eq:int_jump} with $\lambda^1 = 0.1$. We again start with $\hat{y} = 0$. Similarly as before we obtain the point $(x, y) = (\sqrt{0.1}, 0)$ by solving the corresponding NLP subproblem with $UB = -\sqrt{0.1}$. If we then reuse the set of linearization points from $\lambda^0$ and add the point $(\sqrt{0.1}, 0)$ we get $\X^1 = \{(2/3, 0), (\sqrt{0.1}, 0)\}$. The linearizations are plotted in blue in Figure \ref{fig:int_jump_t1}. The optimal solution of the MILP subproblem is now $(x, y) = (2/3, 1)$ showing that the integer variable has changed to another value even when $y = 0$ is the optimal solution for both ($P_{\lambda^0}$) and ($P_{\lambda^1}$). The same effect is observed for any $\lambda^1 \in [0, \lambda^0)$.
\end{example}

Theorem \ref{thm:stability} and \ref{thm:lambda_stable} are the main theoretical results of this paper. Theorem \ref{thm:stability} shows that, under certain conditions, only one iteration is needed for the OA algorithm to converge when warm-starting the algorithm. Theorem \ref{thm:lambda_stable} provides conditions for the problem to ensure convergence in one iteration.

\begin{theorem}\label{thm:stability}
    Let $\lambda^0, \lambda^1 \in P$. Suppose that $(x^0, y^0)$ is an optimal solution to ($P_{\lambda^0}$). Assume that it was found using the OA algorithm with $\Omega(\X^0, \lambda^0)$ as the linear relaxation, meaning that $(x^0, y^0) \in \X^0$. To solve ($P_{\lambda^1}$), assume that the OA algorithm starts with $y^0$ as the integer assignment and let $x^1$ be an optimal solution to (NLP($y^0, \lambda^1$)). Let $\X^1 = \X^0 \cup \{(x^1, y^0)\}$. If $(\hat{x}, y^0)$ is an optimal solution to (MILP($\X^1, \lambda^1$)), then the point $(x^1, y^0)$ is optimal to ($P_{\lambda^1}$). This means that the OA algorithm has converged.
\end{theorem}
\begin{proof}
    We will show that $(\hat{x}, y^0)$ must have the same objective value as $(x^1, y^0)$.
    
    Since $(x^1, y^0) \in \X^1$ we have for all $j \in [p]$ that
    \begin{equation*}
        g^{\lambda^1}(x^1, y^0) + \nabla g^{\lambda^1}(x^1, y^0)^\top
        \begin{pmatrix}
            \hat{x} - x^1 \\
            y^0 - y^0
        \end{pmatrix} \leq 0.
    \end{equation*}
    In particular, for all active constraints $g^{\lambda^1}_j(x^1, y^0) = 0$, i.e., for all $j \in I^{\lambda^1}(x^1, y^0)$, we have that
    \begin{equation*}
        \nabla_x g^{\lambda^1}_j (x^1, y^0)^\top (\hat{x} - x^1) \leq 0.
    \end{equation*}
    In the same manner, it holds that for all active constraints $h_j(x^1, y^0) = 0$, i.e., for all $j \in J(x^1, y^0)$, that
    \begin{equation*}
        \nabla_x h_j(x^1, y^0)^\top (\hat{x} - x^1) \leq 0.
    \end{equation*}
    Finally we have for each $j \in K(x^1, y^0)$ that
    \begin{equation*}
        \begin{cases}
            (a^j)^\top \hat{x} + (b^j)^\top y^0 & \leq b \\
            (a^j)^\top x^1 + (b^j)^\top y^0 & = b,
        \end{cases}
    \end{equation*}
    since $(\hat{x}, y^0)$ is feasible for (MILP($\X^1, \lambda^1$)). Subtracting the second equation from the first we get that $(a^j)^\top (\hat{x} - x^1) \leq 0$ for each $j \in K(x^1, y^0)$.

    It follows that $\hat{x} - x^1 \in \mathcal{L}_{y^0, \lambda^1}(x^1)$ (recall Definition \ref{def:linearized_set}). Since we assume that the LICQ holds for (NLP($y^0$, $\lambda^1$)) in $(x^1, y^0)$, it follows that $\mathcal{L}_{y^0, \lambda^1}(x^1) = \T_{y^0, \lambda^1}(x^1)$ (recall Remark \ref{remark:licq}). Since $x^1$ is an optimal solution to (NLP($y^0$, $\lambda^1$)), Proposition \ref{prop:1st_order_opt} implies that $c^\top (\hat{x} - x^1) \geq 0$. We therefore obtain that
    \begin{equation}\label{eq:nlp_ineq}
        \begin{pmatrix} c \\ d \end{pmatrix}^\top \begin{pmatrix} \hat{x} \\ y^0 \end{pmatrix} \geq \begin{pmatrix} c \\ d \end{pmatrix}^\top \begin{pmatrix} x^1 \\ y^0 \end{pmatrix}.
    \end{equation}
    
    We also know that $(x^1, y^0)$ is feasible for (MILP($\X^1$, $\lambda^1$)), since $(x^1, y^0)$ is feasible for ($P_{\lambda^1}$), meaning that
    \begin{equation}\label{eq:milp_ineq}
        \begin{pmatrix} c \\ d \end{pmatrix}^\top \begin{pmatrix} x^1 \\ y^0 \end{pmatrix} \geq \begin{pmatrix} c \\ d \end{pmatrix}^\top \begin{pmatrix} \hat{x} \\ y^0 \end{pmatrix},
    \end{equation}
    since we assumed that $(\hat{x}, y^0)$ is an optimal solution.

    Therefore, equation \eqref{eq:nlp_ineq} and \eqref{eq:milp_ineq} imply that the objective values of the points $(x^1, y^0)$ and $(\hat{x}, y^0)$ are equal. Since the point $(\hat{x}, y^0)$ is an optimal solution to (MILP($\X^1, \lambda^1$)), that optimal value is a lower bound of the optimal value of ($P_{\lambda^1}$). On the other hand, the point $(x^1, y^0)$ was constructed from the NLP subproblem, meaning that the optimal value of the subproblem is an upper bound of the optimal value of ($P_{\lambda^1}$) and that $(x^1, y^0)$ is feasible for ($P_{\lambda^1}$). Since they are equal, we draw the conclusion that $(x^1, y^0)$ is an optimal solution to ($P_{\lambda^1}$). If $(\hat{x}, y^0)$ is feasible for ($P_{\lambda^1}$), then it is also optimal but not necessarily with $\hat{x} = x^1$ as the minimizer is not necessarily unique.
\end{proof}

The discrete nature of MINLP problems complicates the analysis of how a change in $\lambda$ affects the optimal solution, i.e., a small change in $\lambda$ can result in a large distance between the optimal solutions as illustrated in Example \ref{ex:int_jump}. In Theorem \ref{thm:stability} it was proved that if the integer part of the solution of the MILP subproblem remains the same, then the OA algorithm will converge in one iteration. We now want to investigate what conditions need to hold for this situation to occur. To perform this analysis, we first introduce the concept of a parameter stable optimum.
\begin{definition}
 Let $\lambda \in P$ and $(x^*,y^*)$ be an optimal solution to ($P_{\lambda}$). We say that $(x^*,y^*)$ is a parameter stable optimum for ($P_{\lambda}$) if there exists $\delta > 0$ such that for all $\bar{\lambda} \in B(\lambda, \delta) \cap P$ there is an $x \in \R^n$ such that $(x, y^*)$ is an optimal solution to ($P_{\bar{\lambda}}$).
\end{definition}

We also define what we mean by an integer unique optimal solution.
\begin{definition}\label{def:integer_unique_opt}
  Let $\lambda \in P$ and $\X$ be a set of linearization points. Let $(x^*,y^*)$ be an optimal solution to MILP($\X, \lambda$). We say that $(x^*,y^*)$ is an integer unique optimal solution for MILP($\X, \lambda$) if there exists $\varepsilon > 0$ such that $$c^\top x + d^\top y > c^\top x^* + d^\top y^* + \varepsilon$$ for any feasible point $(x, y)$ where $y \neq y^*$.
\end{definition}

In Example \ref{ex:int_jump}, the point $(2/3, 0)$ was trivially a parameter stable optimum for ($P_{\lambda = 4/9}$) since $y = 0$ is the only feasible integer assignment.  On the other hand, the point $(2/3, 0)$ was not an integer unique optimal solution for (MILP($\X^0, \lambda^0$)) since the point $(2/3, 1)$ has the same objective function value while having a different integer part.

It should be noted that both of the concepts \textit{parameter stable optimum} and \textit{integer uniqueness} are defined in relation to the integer part of the optimal solution. For a parameter stable optimum this means that for a small change in $\lambda$ the continuous part can be adjusted such that $(\bar{x}, y^*)$ is an optimal solution. Similarly, for integer uniqueness, even if an optimal solution $(x^*, y^*)$ is integer unique there can be some $\bar{x} \neq x^*$ such that $(\bar{x}, y^*)$ is an optimal solution.

Lastly, we need a subset of the ball $B(x, \delta)$ for the formulation of Theorem \ref{thm:lambda_stable}.
\begin{definition}
    Let $x \in \R^n$ and $\delta > 0$, then $$H(x, \delta) := \{u \in \R^n \mid u \in B(x, \delta),\ u \leq x \}.$$
\end{definition}

We now have all the components needed to show that warm-starting can, in some settings, reduce the number of iterations the OA algorithm needs to converge to only one iteration. We assume in Theorem \ref{thm:lambda_stable} that the parameterized constraint is of the form $g^\lambda(x, y) = g(x, y) - F\lambda$ where $F$ is a non-negative matrix, which gives a certain structure to the change in the problems. Note that this structure is exactly what is obtained for scalarizations of multiobjective problems, e.g., in the case of the $\varepsilon$-constraint method where $F = I$. In Lemma \ref{lem:motone_omega} we show that if $g^\lambda$ has this form, the linear relaxations have a strict relationship enabling the analysis in Theorem \ref{thm:lambda_stable}. Note that the only difference between $\Omega(\X,\lambda^1)$ and $\Omega(\X,\lambda^0)$ is that $F(\lambda^0 - \lambda^1)$ is subtracted from the right-hand side of the gradient cuts in $\Omega(\X,\lambda^1)$. Thus, $\Omega(\X,\lambda^1)$ can be viewed as a reduction of $\Omega(\X,\lambda^0)$ where some of the half-planes defining the set have been moved inwards. We focus on the case when $\lambda^1 \leq \lambda^0$, but similar results for the opposite can be proven through similar steps.

\begin{lemma}\label{lem:motone_omega}
    Assume that $g^\lambda$ is of the form $g^\lambda(x, y) = g(x, y) - F\lambda$ where $F \in \R^{p \times k}$ is non-negative. Let $\lambda^0, \lambda^1 \in P$ be such that $\lambda^1 \leq \lambda^0$, then $\Omega(\X, \lambda^1) \subset \Omega(\X, \lambda^0)$.
\end{lemma}
\begin{proof}
    Let $f^j$ be the $j$th row of $F$ expressed as a column vector. Then, for each $(\bar{x}, \bar{y}) \in \X$ and $j \in [p]$, it holds that
    \begin{equation*}
        g_j^\lambda (\bar{x}, \bar{y}) + \nabla g_j^\lambda (\bar{x}, \bar{y})^\top \begin{pmatrix}
            x - \bar{x} \\ y - \bar{y}
        \end{pmatrix}
        = g_j (\bar{x}, \bar{y}) + \nabla g_j (\bar{x}, \bar{y})^\top \begin{pmatrix}
            x - \bar{x} \\ y - \bar{y} \end{pmatrix} - (f^j)^\top \lambda.
    \end{equation*}

    Let $(x, y) \in \Omega(\X, \lambda^1)$. We have that
    \begin{multline*}
        0 \geq g_j (\bar{x}, \bar{y}) + \nabla g_j (\bar{x}, \bar{y})^\top \begin{pmatrix}
            x - \bar{x} \\ y - \bar{y} \end{pmatrix} - (f^j)^\top \lambda^1 \\ \geq g_j (\bar{x}, \bar{y}) + \nabla g_j (\bar{x}, \bar{y})^\top \begin{pmatrix}
            x - \bar{x} \\ y - \bar{y} \end{pmatrix} - (f^j)^\top \lambda^0
    \end{multline*}
    since all elements of $F$ are non-negative. Therefore $(x, y) \in \Omega(\X, \lambda^0)$.
\end{proof}

\begin{theorem} \label{thm:lambda_stable}
Let $\lambda^0 \in P$. Assume that $(x^*,y^*)$ is a parameter stable optimal solution to ($P_{\lambda^0}$), and let $\X$ contain all the linearization points used to solve the problem by the OA algorithm. Furthermore, we assume that $(x^*,y^*)$ is an integer unique optimal solution to (MILP($\X$, $\lambda^0$)) and that $g^\lambda$ is of the form $g^\lambda(x, y) = g(x, y) - F \lambda$ where $F$ is a non-negative matrix of appropriate size. Then there exists a $\delta > 0$, such that for any $\lambda^1 \in H(\lambda^0, \delta) \cap P$, the OA algorithm warm-started with the linearization points $\X$ finds an optimal solution to ($P_{\lambda^1}$) and terminates after a single iteration of the OA algorithm.
\end{theorem}

\begin{proof}
Since $(x^*,y^*)$ is a parameter stable optimal solution to ($P_{\lambda^0}$), we know that there exists $\delta_1>0$ such that for any $\lambda^1 \in B(\lambda^0, \delta_1) \cap P$ there exists some $x \in \R^n$ such that $(x, y^*)$ is an optimal solution to ($P_{\lambda^1}$). In particular, let $\lambda^1 \in H(\lambda^0, \delta_1) \cap P$. As $(x^*,y^*)$ is an integer unique optimal solution to (MILP($\X$, $\lambda^0$)), it holds that for any point $(x, y) \in \Omega(\X, \lambda^0)$ where $y \neq y^*$ that $c^\top x + d^\top y > c^\top x^* + d^\top y^* + \varepsilon$ for some $\varepsilon > 0$.

Now consider (NLP($y^*, \lambda^1$)) which is parameterized by $\lambda^1$ and is feasible for $\lambda^1 \in H(\lambda^0, \delta_1) \cap P$. As all functions defining the problem are continuously differentiable, the optimal value as a function of $\lambda$ is continuous (see Proposition 2 in \cite{fiacco_2008}). Therefore, we can choose $\delta_2 > 0$ sufficiently small such that the minimal value of (NLP($y^*$, $\lambda^1$)) is arbitrarily close to $c^\top x^* + d^\top y^*$ for any $\lambda^1 \in H(\lambda^0, \delta_2) \cap P$. We then select a $\delta$, satisfying $0 < \delta < \min\{\delta_1, \delta_2\}$, such that the minimal value of (NLP($y^*$, $\lambda^1$)) is strictly less than $c^\top x^* + d^\top y^* + \varepsilon$ for all $\lambda^1 \in H(\lambda^0, \delta) \cap P$.

Now let $\lambda^1 \in H(\lambda^0, \delta) \cap P$ and let $x^1$ be an optimal solution to NLP($y^*, \lambda^1$). From the construction of $\delta$ we have that $c^\top x^1 + d^\top y^* < c^\top x^* + d^\top y^* + \varepsilon$. By Lemma \ref{lem:motone_omega} it holds that $\Omega(\X,\lambda^1) \subset \Omega(\X,\lambda^0)$. We know that for every point $(x, y) \in \Omega(\X, \lambda^1) \subset \Omega(\X, \lambda^0)$ such that $y \neq y^*$ that $$c^\top x^1 + d^\top y^* < c^\top x^* + d^\top y^* + \varepsilon < c^\top x + d^\top y.$$ Therefore, since $(x^1, y^*) \in \F_{\lambda^1} \subset \Omega(\X,\lambda^1)$, the optimal solution of MILP($\X, \lambda^1$) is of the form $(x, y^*)$ for some $x$. Let $\X^1 = \X \cup \{(x^1, y^*)\}$. The optimal solution of MILP($\X^1, \lambda^1$) is therefore also of the form $(x, y^*)$ for some $x$ since $(x^1, y^*) \in \Omega(\X^1, \lambda^1)$. From Theorem \ref{thm:stability}, it follows that the OA algorithm would terminate after the first iteration with the upper bound equal to the lower bound.
\end{proof}

The assumption of an integer unique solution implies that no symmetry is present, which for the class of problems considered is a strong assumption. We argue, however, that similar results cannot be expected without this assumption. For example, if two points are optimal solutions then a good approximation of the feasible set needs to be obtained around both points before optimality can be proven. It is therefore unlikely that only one iteration of the OA algorithm can be sufficient to converge to an optimal solution.

Theorem~\ref{thm:lambda_stable} serves as an important motivation for the warm-starting strategies proposed in the next section, namely, in some circumstances, a simple warm-starting approach can make the OA algorithm converge in one iteration.

\section{Warm-starting algorithm}\label{sec:algorithm}
This section describes the general algorithm proposed to solve a parametrized MINLP using OA. It uses the OA algorithm to solve \eqref{eq:p_lambda} for a fixed value of $\lambda$ and then uses the information collected to warm-start \eqref{eq:p_lambda} for a new value of $\lambda$. The overall algorithm is described in Algorithm \ref{algo:param_OA}.

\begin{algorithm}[t]
    \caption{Parameterized MINLP}\label{algo:param_OA}
    \begin{algorithmic}
        \Require $\{\lambda^0, \lambda^1, \ldots, \lambda^K\} \subset \R^k$, and $\hat{y} \in \Z^m$ are given.
        \State Set $\X^* = \varnothing$. \Comment{Set of optimal solutions}
        \State Set $\X = \varnothing$. \Comment{Set of linearization points}
        \For{$k \in \{0, 1, \ldots, K\}$}
            \State Solve ($P_{\lambda^k}$) using Algorithm \ref{algo:OA} with $y^0 = \hat{y}$ and $\X^0 = \X$ and return $(x^*, y^*), \hat{\X}$ 
            \State $\X^* \gets \X^* \cup \{(x^*, y^*)\}$
            \State Determine $\X, \hat{y}$ from initialization rule \Comment{Use initialization rule 1, 2 or 3.}
        \EndFor
        \State \Return $\X^*$
    \end{algorithmic}
\end{algorithm}

As a baseline for comparison we use the common initialization of solving the continuous relaxation of \eqref{eq:p_lambda}. This method uses no prior knowledge and is therefore not a warm-starting technique. The initialization is specified in Subsection \ref{sec:relaxation}.

Three initialization rules are presented for warm-starting. The first rule is simply that we warm-start with the optimal solution from the last value of $\lambda$. We consider this the obvious way of warm-starting an algorithm and, therefore, do not regard this as a new concept in relation to existing literature. The second and third rules use properties of the OA algorithm and, as far as we are aware, have not been presented in the literature before. We call the rules cut-tightening and point-based. The cut-tightening rule focuses on updating the cuts with the new parameter value without changing the points of linearization. The point-based rule focuses on updating the points of linearization. The rules are presented in more detail in Subsection \ref{sec:rule0}, \ref{sec:rule1} and \ref{sec:rule2}.

\subsection{Baseline: Relaxation}\label{sec:relaxation}
A common method of initializing the OA algorithm is to solve the continuous relaxation. The solution given in that case is generally not integer feasible of course, but we can still linearize the nonlinear constraints at this point and add the cuts to the MILP subproblem. If the MILP subproblem obtained is then solved, we obtain an integer solution and can start the OA algorithm as usual. This is then used to start the algorithm for each $\lambda$, meaning that no prior knowledge is used. This will be used as the baseline to compare the other methods. The rule is described in Algorithm \ref{algo:relaxation}.

\begin{algorithm}
    \caption{Baseline rule}\label{algo:relaxation}
    \begin{algorithmic}
        \State Solve the continuous relaxation of \eqref{eq:p_lambda} and return $(\bar{x}, \bar{y})$.
        \State Set $\X = \{(\bar{x}, \bar{y})\}$.
        \State Solve (MILP($\X, \lambda$)) and return $\hat{y}$.
        \State \Return $\X$, $\hat{y}$
    \end{algorithmic}
\end{algorithm}

\subsection{Initialization rule 1: Restart for every parameter value}\label{sec:rule0}
The simplest method is to restart the OA algorithm for every value of $\lambda$. This is implemented by using the optimal solution for $\lambda^k$ as the starting point for the OA algorithm for the next value $\lambda^{k+1}$. This means that the initialization rule will be to set $\hat{y} = y^*$ and $\hat{\X} = \varnothing$. The rule is formalized in Algorithm \ref{algo:rule1}.

\begin{algorithm}
    \caption{Initialization rule 1}\label{algo:rule1}
    \begin{algorithmic}
        \Require $\hat{\X}$ and $y^*$ are given.
        \State \Return $\X = \varnothing$, $\hat{y} = y^*$
    \end{algorithmic}
\end{algorithm}

\subsection{Initialization rule 2: Cut tightening}\label{sec:rule1}
By Proposition \ref{prop:valid_oa} we know that any set of points $\X$ produces a linear relaxation of the feasible set of ($P_\lambda$) where the cuts depend on the parameter $\lambda$. A computationally cheap way to start the OA algorithm for a new parameter value is therefore to start with the same linearization points $\X$ but update the parameter $\lambda$. The initialization rule is therefore $\hat{y} = y^*$ and $\hat{\X} = \X^k$. It should be noted that even though the points of linearization are the same, the cuts are in general not the same since they depend on $\lambda$. The intuition behind this method is that as long as the change in the parameter is small, the change in the parameterized problem is also small and similar cuts should be needed to solve the problem. This is also confirmed by the theoretical results in Theorem \ref{thm:stability} and \ref{thm:lambda_stable} where it was proved that under some conditions a small enough change in $\lambda$ results in convergence in one iteration of the OA algorithm. Therefore, the same set of linearization points can approximate the new feasible set in a satisfactory way. The rule is formalized in Algorithm \ref{algo:rule2}.

\begin{algorithm}
    \caption{Initialization rule 2}\label{algo:rule2}
    \begin{algorithmic}
        \Require $\hat{\X}$ and $y^*$ are given.
        \State \Return $\X = \hat{\X}$, $\hat{y} = y^*$ 
    \end{algorithmic}
\end{algorithm}

\subsection{Initialization rule 3: Point based}\label{sec:rule2}
The cut tightening method is cheap to compute, but since the points are not adapted to the problem the cuts are in general not tight to the feasible set of \eqref{eq:p_lambda}. Furthermore, solving the MILP subproblems in order to find integer assignments is usually the bottleneck for the OA algorithm. This means that already having a list of candidate integer assignments could make it more efficient since we do not need to find integer assignments again while being able to construct tight cuts.

The point-based method attempts to use this observation. Since we assume that the change in the feasible set is small, it seems likely that the same integer assignments will be explored for the algorithm to converge. We therefore reuse the integer assignments that were already used earlier in the sequence of problems. In the OA algorithm, for a given integer combination $y$ we find the corresponding $x$ by solving (NLP($y, \lambda$)), or (NLP-Feas($y, \lambda$)) if $y$ is not feasible. For some value $\lambda^0$ the OA algorithm will return a set of points $\X^0 = \{(x^1, y^1), \ldots, (x^l, y^l)\}$. If the next value of the parameter is $\lambda^1$, we can for each $y^i$ solve (NLP($y^i, \lambda^1$)) or (NLP-Feas($y^i, \lambda^1$)) to get $\hat{x}^i$. We have then constructed the set of points $\X^1$. The initialization rule is then to use $\hat{\X} = \X^1$ and $\hat{y} = y^*$ which is formalized in Algorithm \ref{algo:rule3}.

This warm-starting will be more computationally demanding than the cut tightening method due to additional NLP problems. The hope is that the computations will be worthwhile if they result in a tighter initial outer approximation and reduce the number of iterations needed for the OA algorithm.

\begin{algorithm}
    \caption{Initialization rule 3}\label{algo:rule3}
    \begin{algorithmic}
        \Require $\hat{\X}, y^*$ and $\lambda^{k+1}$ are given.
        \State Set $\X = \varnothing$
        \For{$(x, y) \in \hat{\X}$}
            \If{(NLP($y, \lambda^{k+1}$)) is feasible}
                \State Solve (NLP($y, \lambda^{k+1}$)) with optimal solution $\hat{x}$
            \Else
                \State Solve (NLP-Feas($y, \lambda^{k+1}$)) with optimal solution $\hat{x}$
            \EndIf
            \State $\X \gets \X \cup \{(\hat{x}, y)\}$
        \EndFor
        \State \Return $\X$, $\hat{y} = y^*$
    \end{algorithmic}
\end{algorithm}

\subsection{Illustration of warm-starting rules}\label{sec:rule_illustration}
To illustrate the initialization rules, we will adapt Example \ref{ex:warm-starting}. Consider the case where \eqref{eq:worst-case-warm-starting} is solved with the OA algorithm resulting in linearizations at $y = 2$ and $y = 3$. This is sufficient to prove the optimality of the point $(x, y) = (1.51, 2)$. If the right-hand side of the nonlinear constraint is changed to 5 (from 5.3) we can use the warm-starting rules and compare the linear relaxation obtained. We call this problem ($\ref{eq:worst-case-warm-starting}'$).

The different linear relaxations are plotted in Figure \ref{fig:ex_warm_starting_baseline}. For the baseline rule, the continuous relaxation of ($\ref{eq:worst-case-warm-starting}'$) is solved and the red line is obtained by linearizing at the optimal point. For rule 1, $y = 2$ will be the initial integer part and the point $(1.47, 2)$ is obtained by solving ($\ref{eq:worst-case-warm-starting}'$) with $y$ fixed to 2. The linearization at this point is plotted in blue. For rule 2, the nonlinear constraint is linearized at the points $(1.51, 2)$ and $(0.63, 3)$ to obtain the linearizations plotted in green. Lastly, for rule 3, problem ($\ref{eq:worst-case-warm-starting}'$) is solved with $y$ fixed to 2 and to 3 since that corresponds to the integer parts for solving \eqref{eq:worst-case-warm-starting}. For $y = 2$, the same linearization as for rule 1 is obtained. For $y = 3$, the problem is no longer feasible, therefore the feasibility problem is solved to obtain $x = 0.5$. The nonlinear constraint is linearized at $(0.5, 3)$ to obtain the second linearization in yellow.

\begin{figure}[ht]
    \centering
    \begin{tikzpicture}
        \begin{axis}[
        width = 0.8\textwidth,
        xlabel = $x$,
        ylabel = $y$,
        xmin = -4.5,
        xmax = 4.5,
        ymin = -2.5,
        ymax = 9.5,
        ytick = {-2, -1, 0, 1, 2, 3, 4, 5, 6, 7, 8, 9},
        xtick = {-4, -2, 0, 2, 4},
        axis equal image,
        ylabel style={rotate=-90},
        ymajorgrids=true,
        grid style=dashed,
        set layers = Bowpark,
        colormap = {bw}{color = (black) color = (black)},
        legend pos = outer north east
        ]

        \draw[black] (axis cs:-0.2,0.9) ellipse[
            rotate = 58.28,
            x radius = 2.26662927362,
            y radius = 1.40085393099
            ];

        \addplot[domain = (-2:2),
                samples = 50,
                color = black
                ]
                {4 + 10*x}; \addlegendentry{Constraints}

        \addplot[domain = -2.5:4.5,
                samples = 10,
                color = red]
                {(7.87 - 3.08*x)/1.54}; \addlegendentry{Baseline}

        \addplot[domain = -2.5:2.5,
                samples = 50,
                color = blue
                ]
                {(6.8-3.91*x)/0.53}; \addlegendentry{Rule 1}

        \addplot[domain = -4.5:4.5,
                samples = 50,
                color = green
                ]
                {3 - (0.3007+0.78*(x-0.63))/6.74}; \addlegendentry{Rule 2}

        \addplot[domain = -2.5:2.5,
                samples = 50,
                color = green,
                forget plot
                ]
                {2 - (0.3303+8.06*(x-1.51))/0.98};

        \addplot[domain = -4.5:4.5,
                samples = 50,
                color = yellow
                ]
                {3 - (0.25/7}; \addlegendentry{Rule 3}

        \addplot[domain = -2.5:2.5,
                samples = 50,
                color = yellow,
                dashed
                ]
                {2 - (0.0127 + 7.82*(x-1.47))/1.06};
        \end{axis}
    \end{tikzpicture}

    \caption{Comparison of linear relaxations for the different warm-starting rules.}
    \label{fig:ex_warm_starting_baseline}
\end{figure}

For this simple illustrative problem, we observe that rule 2 and rule 3 result in quite similar relaxations. One can argue that rule 3 results in a slightly better relaxation as it excludes all $y \geq 3$, but keep in mind it comes at a computational cost.

\section{Numerical experiments}\label{sec:num_exp}
\subsection{Test problems}\label{sec:test_probs}
We briefly present the five test problems that have been used to investigate the potential of the warm-starting techniques described above. The test problems were chosen to cover different types of changes due to the $\lambda$ parameters, e.g., right-hand side perturbations of linear or nonlinear constraints and changing parameters in nonlinear expressions. First in Subsection \ref{sec:test_prob_biobj}, two biobjective problems are included for the multiobjective application. The structure of these problems is quite simple since the parameter appears linearly as a change to the right-hand side and does not affect the gradient of the constraints. In Subsection \ref{sec:test_probs_slr}, sparse linear regression is presented which is a useful application where the parameter has a more complex effect than the case of scalarization for multiobjective problems since the gradients also depend on the parameter. Lastly, a problem from Model Predictive Control (MPC) is presented in Subsection \ref{sec:test_probs_mpc}. In this problem the effect of the parameter is even more complex than in the other ones, and involves a vector-valued parameter.

\subsubsection{Biobjective problems}\label{sec:test_prob_biobj}
As mentioned in the introduction, solving multiobjective MINLP by scalarization techniques sparked the initial interest in parameterized MINLP. The methods have, therefore, been implemented for two biobjective problems. Since the problems are biobjective the parametrization only requires a scalar-valued parameter $\lambda$. The two biobjective problems \eqref{eq:ti4} and \eqref{eq:ti14} that have been considered have been adapted from \cite{Biobj_test_probs} and we refer to the paper for more details and plots of the image sets. The problems were chosen to be simple enough to analyze the correctness and details of the algorithm.

The first test problem for a biobjective problem is problem \eqref{eq:ti4} defined as
\begin{equation}\tag{TI4}\label{eq:ti4}
    \begin{array}{rl}
        \displaystyle \min_{x, y} & \begin{pmatrix}
            x_1 + x_3 + y_1 + y_3 \\
            x_2 + x_4 + y_2 + y_4
        \end{pmatrix} \\
        \text{s.t.} & x_1^2 + x_2^2 \leq 1, \\
                    & x_3^2 + x_4^2 \leq 1, \\
                    & (y_1 - 2)^2 + (y_2 - 5)^2 \leq 10, \\
                    & (y_3 - 3)^2 + (y_4 - 8)^2 \leq 10, \\
                    & x, y \in [-20, 20]^4, \\
                    & x \in \R^4, y \in \Z^4.
    \end{array}
\end{equation}
Problem \eqref{eq:ti4} was formulated for the computational method as
\begin{equation*}
    \begin{array}{rl}
        \displaystyle \min_{t, x, y} & t \\
        \text{s.t.} & x_1 + x_3 + y_1 + y_3 \leq t, \\
                    & x_2 + x_4 + y_2 + y_4 \leq \lambda, \\
                    & x_1^2 + x_2^2 \leq 1, \\
                    & x_3^2 + x_4^2 \leq 1, \\
                    & (y_1 - 2)^2 + (y_2 - 5)^2 \leq 10, \\
                    & (y_3 - 3)^2 + (y_4 - 8)^2 \leq 10, \\
                    & x, y \in [-20, 20]^4, \\
                    & t \in \R, x \in \R^4, y \in \Z^4.
    \end{array}
\end{equation*}
The epigraphical reformulation was used to simplify the computational implementation.

The second test problem is \eqref{eq:ti14}, defined as
\begin{equation}\tag{TI14}\label{eq:ti14}
    \begin{array}{rl}
        \displaystyle \min_{x, y} & \begin{pmatrix}
            x_1+x_3+y_1+\exp(y_3) - 1 \\
            x_2 + x_4 + y_2 + y_4
        \end{pmatrix} \\
        \text{s.t.} & x_1^2 + x_2^2 \leq 1, \\
                    & x_3^2 + x_4^2 \leq 1, \\
                    & (y_1-2)^2 + (y_2-5)^2 \leq 10, \\
                    & (y_3-3)^2 + (y_4-8)^2 \leq 10, \\
                    & x, y \in [-20, 20]^4, \\
                    & x \in \R^4, y \in \Z^4.
    \end{array}
\end{equation}
Similarly to \eqref{eq:ti4}, problem \eqref{eq:ti14} was formulated as
\begin{equation*}
    \begin{array}{rl}
        \displaystyle \min_{t, x, y} & t \\
        \text{s.t.} & x_1+x_3+y_1+\exp(y_3) - 1 \leq \lambda, \\
                    & x_2 + x_4 + y_2 + y_4 \leq t, \\
                    & x_1^2 + x_2^2 \leq 1, \\
                    & x_3^2 + x_4^2 \leq 1, \\
                    & (y_1-2)^2 + (y_2-5)^2 \leq 10, \\
                    & (y_3-3)^2 + (y_4-8)^2 \leq 10, \\
                    & x, y \in [-20, 20]^4, \\
                    & t \in \R, x \in \R^4, y \in \Z^4.
    \end{array}
\end{equation*}
Here, the linear objective function in the biobjective problem was used as the objective function for the scalarized formulation due to simpler implementation.

\subsubsection{Sparse linear regression}\label{sec:test_probs_slr}
In sparse linear regression, as presented in \cite{Bertsimas_Parys_2020}, we consider the problem
\begin{equation*}
    \begin{array}{rl}
        \displaystyle \min_{x} & \frac{1}{2} \norm{Ax-b}{2}^2 + \frac{\lambda}{2} \norm{x}{2}^2 \\
        \text{s.t.} & \norm{x}{0} \leq \kappa, \\
        & x \in \R^n.
    \end{array}
\end{equation*}
This is a typical linear regression with an $l_2$-regularization term with the added constraint that only $\kappa$ number of elements of $x$ are allowed to be non-zero. In applications, choosing the values for $\lambda$ or $\kappa$ is not trivial meaning that there is an interest in solving the problem for several values and comparing the results. This results in a parameterized MINLP.

To obtain the regression problem of the form \eqref{eq:p_lambda}, we introduce binary variables $y_i$. We assume that it is possible to define some lower bound $l_i$ and upper bound $u_i$ for each component $x_i$. The regression problem can then be formulated as
\begin{equation}\tag{SLR}\label{eq:slr}
    \begin{array}{rl}
        \displaystyle \min_{t, x, y} & t \\
        \text{s.t.} & \frac{1}{2} \norm{Ax-b}{2}^2 + \frac{\lambda}{2} \norm{x}{2}^2 \leq t, \\
        & l_i y_i \leq x_i \leq u_i y_i, \\
        & \sum_{i=1}^{n} y_i \leq \kappa, \\
        & t \in \R, x \in \R^n \text{, } y \in \{0, 1\}^n.
    \end{array}
\end{equation}

For this problem there are two parameters that can be varied: $\kappa$ and $\lambda$. Numerical results will be reported for fixing one of the parameters and varying the other. However, we hypothesize that the case of varying $\lambda$ will be more suitable for the proposed warm-starting than the case of varying $\kappa$. The underlying assumption and intuition for the proposed methods are that the integer part of points used to construct the linear relaxation is similar for different parameter values. Here, when changing $\kappa$, it is unlikely that the previously obtained integer solutions are of interest.

\subsubsection{Model predictive control}\label{sec:test_probs_mpc}
For MPC, we use the following example from the website for OSQP \cite{mpc}:
\begin{equation*}
    \begin{array}{rl}
        \displaystyle \min_{x, u} & \sum_{i = 1}^n (x^i - x^r)^\top Q (x^i - x^r) + \sum_{i = 0}^{n-1} (u^i)^\top R (u^i) \\
        \text{s.t.} & Ax^i + Bu^i = x^{i+1} \text{ for all } i \in [n-1]_0,\\
        & x^{\min} \leq x^i \leq x^{\max} \text{ for all } i \in [n], \\
        & u^i \in \{u^{\min}, u^{\min}  + \Delta u, \ldots, u^{\max} - \Delta u, u^{\max}\} \text{ for all } i \in [n-1]_0, \\
        & x^0 = \lambda, \\
        & x^i \in \R^{n_x} \text{ for all } i \in [n]_0, \\
        & u^i \in \R^{n_u} \text{ for all } i \in [n-1]_0.
    \end{array}
\end{equation*}
The vectors $x^i$ are the states and the vectors $u^i$ are the control inputs at each time step $i$. The vector $x^r \in \R^{n_x}$ is the reference state of the model and the vector $\lambda \in \R^{n_x}$ is the initial state. The vectors $x^{\min}, x^{\max} \in \R^{n_x}$ and $u^{\min}, u^{\max} \in \R^{n_u}$ are predefined. The real matrices $A, B, Q$ and $R$ are of suitable size. The problem is solved several times, referred to as simulations, to find the control inputs $u^k$ to get as close as possible to the reference state $x^r$. After each simulation a new intial state $\bar{\lambda}$ is found by applying the control to the first state, i.e., $\bar{\lambda} = A\lambda + Bu^0$. This example was adapted by assuming that the control input $u^i$ only can take a discrete set of values. To obtain the problem in the form of \eqref{eq:p_lambda}, it was rewritten as
\begin{equation}\tag{MPC}\label{eq:mpc}
    \begin{array}{rl}
        \displaystyle \min_{t, x, u, y} & t \\
        \text{s.t.} & \sum_{i = 1}^n (x^i - x^r)^\top Q (x^i - x^r) + \sum_{i = 0}^{n-1} (u^i)^\top R u^i \leq t, \\
        & Ax^i + Bu^i = x^{i+1}  \text{ for all } i \in [n-1]_0, \\
        & x^{\min} \leq x^i \leq x^{\max} \text{ for all } i \in [n], \\
        & u^i = y^i \Delta u \text{ for all } i \in [n-1]_0, \\
        & y^{\min}_j \leq y^i_j \leq y^{\max}_j \text{ for all } i \in [n-1]_0 \text{ and } j \in [n_u]\\ 
        & x^0 = \lambda, \\
        & t \in \R, \\
        & x^i \in \R^{n_x} \text{ for all } i \in [n]_0, \\
        & u^i \in \R^{n_u} \text{ for all } i \in [n-1]_0, \\
        & y^i \in \Z^{n_u} \text{ for all } i \in [n-1]_0.
    \end{array}
\end{equation}
The data used was $y^{\min}_j = -1$ and $y^{\max}_j = 4$ for all $j \in [n_u]$. This corresponds to $u^{\min}_j = -0.5$, $\Delta u_j = 0.5$ and $u^{\max}_j = 2$ for all $j \in [n_u]$. The size was set to $n_x = 12$ and $n_u = 4$. For specifications of $Q, R, A, B, x^{\min}$ and $x^{\max}$, see \cite{mpc}.

\subsection{Numerical results}\label{sec:num_res}
To test the proposed methods, the problems from Subsection \ref{sec:test_probs} are solved and we analyze the results. The implementation was done in Julia v. 1.11.1 using JuMP \cite{Julia-2017, Lubin2023}. The OA algorithm has been implemented by the authors to have full control over the algorithm and extraction of any data. It is worth keeping in mind that the implementations are rather simple, and serve as a proof of concept. The code could be further improved to speed up the warm-starting algorithms. For example, the NLP problems in initialization rule 3 could be run in parallel. However, the results still clearly show the potential of the proposed warm-starting techniques.

The MILP subproblems were solved using Gurobi v. 11.0.3 and the NLP subproblems were solved using Ipopt v. 3.14.16. The tests were executed on a MacBook Pro with the Apple M1 Pro chip using up to 10 threads and 16 GB of memory. The tests were conducted with the following parameters:
\begin{itemize}
    \item For \eqref{eq:ti4} the range was $10.5 \geq \lambda \geq 5.1$ with step size $\Delta\lambda = -0.05$.
    \item For \eqref{eq:ti14} the range was $7 \geq \lambda \geq -3$ with steps size $\Delta\lambda = -0.05$.
    \item For \eqref{eq:slr} a dataset on Wine quality of Portuguese red wine \cite{wine_data} was used. The dataset contains 11 features and 1599 data points. The regression problem was solved first by varying $\lambda$ with $\kappa = 5$ and $0 \leq \lambda \leq 20$ with step size $\Delta\lambda = 0.25$ and it was also solved by varying $\kappa$ with $\lambda = 5$ and $1 \leq \kappa \leq 8$ with $\Delta \kappa = 1$.
    \item For \eqref{eq:mpc} the horizon was $n = 4$ and 15 simulations were run.
\end{itemize}

In Subsection \ref{sec:results_overview} an overview of the results is discussed using some summarizing statistics. Then in Subsection \ref{sec:detailed_results}, more detailed analysis is conducted. The results are finally summarized in Subsection \ref{sec:results_summary}.

\subsubsection{Overview of results}\label{sec:results_overview}
The results are summarized in Table~\ref{tab:num_res}. For each problem and initialization rule, the table reports the time taken to solve the instance in seconds and how many NLP and MILP subproblems were solved in total. The number of MILP subproblems solved is the same as the number of iterations of the OA algorithm.

\begin{table}[ht]
    \centering
    \begin{tabular}{l|l|rrrr}
                                            &       & Relaxation & Rule 1    & Rule 2    & Rule 3       \\
        \hline
        \multirow{3}{4em}{\eqref{eq:ti4}}   & Time  & 6 s & 11 s   & \textbf{1 s}    & 34 s              \\
                                            & NLP   & 850 & 1359   & \textbf{124}    & 2368                \\
                                            & MILP  & 551 & 763    & \textbf{119}    & 122                 \\
        \hline
        \multirow{3}{4em}{\eqref{eq:ti14}}  & Time  & 39 s  & 132 s  & \textbf{1 s}    & 50 s            \\
                                            & NLP   & 1473 & 2777   & \textbf{221}    & 4852                \\
                                            & MILP  & 961 & 1530   & \textbf{214}    & 219                 \\
        \hline
        \multirow{3}{4em}{\eqref{eq:slr} vary $\lambda$}   & Time  & 417 s & 212 s  & \textbf{13 s}   & 116 s    \\
                                            & NLP   & 19037 & 7647   & \textbf{243}    & 8321                  \\
                                            & MILP  & 19037 & 7647   & \textbf{243}    & 261                    \\
        \hline
        \multirow{3}{4em}{\eqref{eq:slr} vary $\kappa$}   & Time  & \textbf{16 s} & 19 s  & 30 s   & 43 s          \\
                                            & NLP   & 1036 & 960    & \textbf{635}    & 2397                   \\
                                            & MILP  & 1036 & 960    & 635    & \textbf{633}                     \\
        \hline
        \multirow{3}{4em}{\eqref{eq:mpc}}   & Time  & 995 s & 754 s  & \textbf{365 s}     & 1893 s                  \\
                                            & NLP   & 2288 & 1684   & \textbf{851}    & 13376                   \\
                                            & MILP  & 2288 & 1684   & \textbf{851}    & 1628                    \\
        \hline
    \end{tabular}
    \caption{Numerical results for solving the test problems using Algorithm \ref{algo:param_OA}. Each problem is solved as specified in Subsection \ref{sec:num_res} where time denotes the total time of solving the instance in seconds, NLP denotes the number of NLP subproblems solved and MILP denotes the number of MILP subproblems solved. The number of MILP subproblems also corresponds to the number of iterations for the OA algorithm. The best result per line is marked in bold.}
    \label{tab:num_res}
\end{table}

We note that relaxation and rule 1 have similar performance with some variability depending on the problem. For example for \eqref{eq:ti14} the relaxation performs better, while for \eqref{eq:slr} when varying $\lambda$ rule 1 is better. The situation illustrated in Example \ref{ex:warm-starting} therefore seems to hold in practice for these problems, i.e., warm-starting using only the optimal solution from the last parameter value is not enough to achieve improved performance.

Better performance is observed for rules 2 and 3, which produce noticeably better results compared to relaxation and rule 1. The only exception is the solution time for \eqref{eq:slr} when varying $\kappa$, where relaxation is the fastest. In general, rule 2 seems efficient with fast solution times and fewer subproblems. The major downside of rule 3 is that many NLP subproblems need to be solved while not decreasing the number of MILP subproblems. The expectation beforehand was that rule 3 would result in fewer iterations to convergence than rule 2 since the linear relaxation is more adapted to the updated feasible set as the linearization points are optimal to (NLP($\hat{y}$)) or (NLP-Feas($\hat{y}$)). This was, however, not the case since roughly as many MILP subproblems are solved for rules 2 and 3. The extra computational work in rule 3 therefore does not seem to pay off in terms of fewer iterations of the OA algorithm. It is not completely clear why this happens, but one possible explanation is that rule 2 approximates the feasible set well enough when the integer part of the optimal solution is constant and when the integer part of the optimal solution changes, rule 3 does not adapt better than rule 2 resulting in a similar number of iterations for both rules.

It should be noted here that the results of rule 3 in some cases can be improved. By considering Algorithm \ref{algo:rule3}, we see that if we know whether an integer point is feasible or infeasible we can reduce the number of NLP subproblems needed to be solved. But even if the number of NLP subproblems solved is halved, it is still notably more than for rule 2 and it will of course not affect the number of MILP subproblems needed to be solved.

\subsubsection{Detailed results}\label{sec:detailed_results}
\pgfplotsset{ every non boxed x axis/.append style={x axis line style=-},
     every non boxed y axis/.append style={y axis line style=-}}

\begin{figure}[p]
    \centering
    \begin{subfigure}[t]{\textwidth}
        \centering
        \begin{tikzpicture}
            \begin{axis}[enlargelimits = false,
                        legend pos=north west,
                        width = 0.65*\textwidth,
                        height = 0.5\textwidth,
                        ymin = -1,
                        xlabel = $\lambda$,
                        ylabel = Iterations,
                        x dir = reverse,
                        legend pos=outer north east]
                \addplot[black, solid] table[x = lambda, y = relaxation]{data_files/oa_iters_biobj1.dat}; \addlegendentry{Relaxation}
                \addplot[red, dashed] table[x = lambda, y = rule1]{data_files/oa_iters_biobj1.dat}; \addlegendentry{Rule 1}
                \addplot[blue, dashdotted] table[x = lambda, y = rule2]{data_files/oa_iters_biobj1.dat}; \addlegendentry{Rule 2}
                \addplot[green, dashdotdotted] table[x = lambda, y = rule3]{data_files/oa_iters_biobj1.dat}; \addlegendentry{Rule 3}
            \end{axis}
        \end{tikzpicture}
        \caption{Number of iterations for the OA algorithm to converge for \eqref{eq:ti4}.}
        \label{fig:oa_iters_ti4}
    \end{subfigure}
    \begin{subfigure}[c]{\textwidth}
        \centering
        \begin{tikzpicture}
            \begin{axis}[enlargelimits = false,
                        legend pos=north west,
                        width = 0.65*\textwidth,
                        height = 0.5\textwidth,
                        ymin = 1,
                        x dir = reverse,
                        ymode = log,
                        ytick = {1, 15, 150},
                        yticklabels = {1, 15, 150},
                        xlabel = $\lambda$,
                        ylabel = Constraints,
                        legend pos=outer north east]
                \addplot[black, solid] table[x = lambda, y = relaxation]{data_files/nCons_biobj1.dat}; \addlegendentry{Relaxation}
                \addplot[black, densely dotted] table[x = lambda, y = relaxation]{data_files/ncons_presolved_biobj1_relaxation_lambda.dat}; \addlegendentry{Relaxation - Presolved}
                \addplot[red, loosely dashed] table[x = lambda, y = rule1]{data_files/nCons_biobj1.dat}; \addlegendentry{Rule 1}
                \addplot[red, densely dashed] table[x = iter, y = restart]{data_files/ncons_presolved_biobj1_restart_lambda.dat}; \addlegendentry{Rule 1 - Presolved}
                \addplot[blue, loosely dashdotted] table[x = lambda, y = rule2]{data_files/nCons_biobj1.dat}; \addlegendentry{Rule 2}
                \addplot[blue, densely dashdotted] table[x = iter, y = cuts]{data_files/ncons_presolved_biobj1_cuts_lambda.dat}; \addlegendentry{Rule 2 - Presolved}
                \addplot[green, loosely dashdotdotted] table[x = lambda, y = rule2]{data_files/nCons_biobj1.dat}; \addlegendentry{Rule 3}
                \addplot[green, densely dashdotdotted] table[x = iter, y = points]{data_files/ncons_presolved_biobj1_points_lambda.dat}; \addlegendentry{Rule 3 - Presolved}
            \end{axis}
        \end{tikzpicture}
        \caption{Number of constraints in the subproblem when the OA algorithm has converged for \eqref{eq:ti4}. Note the logarithmic scale of the vertical axis.}
        \label{fig:ncons_ti4}
    \end{subfigure}
    \begin{subfigure}[b]{\textwidth}
        \centering
        \begin{tikzpicture}[every mark/.append style={mark size=1pt}]
            \begin{groupplot}[
            group style =   {group name = my plots,
                            group size = 1 by 2,
                            xticklabels at = edge bottom,
                            ylabels at = edge left,
                            yticklabels at = edge left,
                            vertical sep = 0pt},
                            width = 0.65\textwidth,
                            height = 0.3\textwidth,
                            x dir = reverse,
                            enlargelimits = false]
            \nextgroupplot  [ymin = -4,
                            legend pos = outer north east,
                            axis x line = top,
                            only marks]
                \addplot[black, mark = o] table[x = lambda, y = UB]{data_files/biobj_UB_relax.dat}; \addlegendentry{UB - Relaxation}
                \addplot[black, mark = square] table[x = lambda, y = LB]{data_files/biobj_LB_after_relax.dat}; \addlegendentry{LB - Relaxation}
                \addplot[magenta, mark = o] table[x = lambda, y = UB]{data_files/biobj_UB_restart.dat}; \addlegendentry{UB - Rule 1, 2 \& 3}
                \addplot[red, mark = square] table[x = lambda, y = LB]{data_files/biobj_LB_after_restart.dat}; \addlegendentry{LB - Rule 1}
                \addplot[blue, mark = square] table[x = lambda, y = LB]{data_files/biobj_LB_before_cuts.dat}; \addlegendentry{LB - Rule 2}
                \addplot[green, mark = square] table[x = lambda, y = LB]{data_files/biobj_LB_before_point.dat}; \addlegendentry{LB - Rule 3}
                \coordinate (top) at (rel axis cs:0,1);
            \nextgroupplot[ymax = -4, xlabel = $\lambda$, axis x line = bottom, only marks]
                \addplot[black, mark = o] table[x = lambda, y = UB]{data_files/biobj_UB_relax.dat};
                \addplot[black, mark = square] table[x = lambda, y = LB]{data_files/biobj_LB_after_relax.dat};
                \addplot[magenta, mark = o] table[x = lambda, y = UB]{data_files/biobj_UB_restart.dat};
                \addplot[red, mark = square] table[x = lambda, y = LB]{data_files/biobj_LB_after_restart.dat};
                \addplot[blue, mark = square] table[x = lambda, y = LB]{data_files/biobj_LB_before_cuts.dat};
                \addplot[green, mark = square] table[x = lambda, y = LB]{data_files/biobj_LB_before_point.dat};
                \coordinate (bot) at (rel axis cs:1,0);
            \end{groupplot}
            \path (top-|current bounding box.west)--node[anchor=south,rotate=90] {Bounds} (bot-|current bounding box.west);
        \end{tikzpicture}
        \caption{Initial lower and upper bounds using different initialization rules for \eqref{eq:ti4}. Note the scaling of the vertical axis.}
        \label{fig:LB-UB_ti4}
    \end{subfigure}
    \caption{Results for solving \eqref{eq:ti4}.}
    \label{fig:plots_ti4}
\end{figure}

\begin{figure}[p]
    \centering
    \begin{subfigure}[t]{\textwidth}
        \centering
        \begin{tikzpicture}
            \begin{axis}[enlargelimits = false,
                        legend pos=north west,
                        ymin = -1,
                        width = 0.65\textwidth,
                        height = 0.5\textwidth,
                        xlabel = $\lambda$,
                        ylabel = Iterations,
                        x dir = reverse,
                        legend pos=outer north east]
                \addplot[black, solid] table[x = lambda, y = relaxation]{data_files/oa_iters_biobj2.dat}; \addlegendentry{Relaxation}
                \addplot[red, dashed] table[x = lambda, y = rule1]{data_files/oa_iters_biobj2.dat}; \addlegendentry{Rule 1}
                \addplot[blue, dashdotted] table[x = lambda, y = rule2]{data_files/oa_iters_biobj2.dat}; \addlegendentry{Rule 2}
                \addplot[green, dashdotdotted] table[x = lambda, y = rule3]{data_files/oa_iters_biobj2.dat}; \addlegendentry{Rule 3}
            \end{axis}
        \end{tikzpicture}
        \caption{Number of iterations for the OA algorithm to converge for \eqref{eq:ti14}.}
        \label{fig:oa_iters_ti14}
    \end{subfigure}
    \begin{subfigure}[c]{\textwidth}
        \centering
        \begin{tikzpicture}
            \begin{axis}[enlargelimits = false,
                        legend pos=north west,
                        ymin = 1,
                        width = 0.65\textwidth,
                        height = 0.5\textwidth,
                        x dir = reverse,
                        ymode = log,
                        ytick = {1, 15, 150},
                        yticklabels = {1, 15, 150},
                        ylabel = Constraints,
                        xlabel = $\lambda$,
                        legend pos=outer north east]
                \addplot[black, solid] table[x = lambda, y = relaxation]{data_files/nCons_biobj2.dat}; \addlegendentry{Relaxation}
                \addplot[black, densely dotted] table[x = lambda, y = relaxation]{data_files/ncons_presolved_biobj2_relaxation_lambda.dat}; \addlegendentry{Relaxation - Presolved}
                \addplot[red, loosely dashed] table[x = lambda, y = rule1]{data_files/nCons_biobj2.dat}; \addlegendentry{Rule 1}
                \addplot[red, densely dashed] table[x = iter, y = restart]{data_files/ncons_presolved_biobj2_restart_lambda.dat}; \addlegendentry{Rule 1 - Presolved}
                \addplot[blue, loosely dashdotted] table[x = lambda, y = rule2]{data_files/nCons_biobj2.dat}; \addlegendentry{Rule 2}
                \addplot[blue, densely dashdotted] table[x = iter, y = cuts]{data_files/ncons_presolved_biobj2_cuts_lambda.dat}; \addlegendentry{Rule 2 - Presolved}
                \addplot[green, loosely dashdotdotted] table[x = lambda, y = rule3]{data_files/nCons_biobj2.dat}; \addlegendentry{Rule 3}
                \addplot[green, densely dashdotdotted] table[x = iter, y = points]{data_files/ncons_presolved_biobj2_points_lambda.dat}; \addlegendentry{Rule 3 - Presolved}
            \end{axis}
        \end{tikzpicture}
        \caption{Number of constraints in the subproblem when the OA algorithm has converged for \eqref{eq:ti14}. Note that the vertical axis has a logarithmic scale.}
        \label{fig:ncons:ti14}
    \end{subfigure}
    \begin{subfigure}[b]{\textwidth}
        \centering
        \begin{tikzpicture}[every mark/.append style={mark size=1pt}]
            \begin{groupplot}[
            group style =   {group name = my plots,
                            group size = 1 by 2,
                            xticklabels at = edge bottom,
                            ylabels at = edge left,
                            yticklabels at = edge left,
                            vertical sep = 0pt},
                            width = 0.65\textwidth,
                            height = 0.3\textwidth,
                            x dir = reverse,
                            enlargelimits = false]
            \nextgroupplot  [ymin = 3,
                            legend pos = outer north east,
                            axis x line = top,
                            only marks]
                \addplot[black, mark = o] table[x = lambda, y = UB]{data_files/biobj2_UB_relax.dat}; \addlegendentry{UB - Relaxation}
                \addplot[black, mark = square] table[x = lambda, y = LB]{data_files/biobj2_LB_after_relax.dat}; \addlegendentry{LB - Relaxation}
                \addplot[magenta, mark = o] table[x = lambda, y = UB]{data_files/biobj2_UB_restart.dat}; \addlegendentry{UB - Rule 1, 2 \& 3}
                \addplot[red, mark = square] table[x = lambda, y = LB]{data_files/biobj2_LB_after_restart.dat}; \addlegendentry{LB - Rule 1}
                \addplot[blue, mark = square] table[x = lambda, y = LB]{data_files/biobj2_LB_before_cuts.dat}; \addlegendentry{LB - Rule 2}
                \addplot[green, mark = square] table[x = lambda, y = LB]{data_files/biobj2_LB_before_point.dat}; \addlegendentry{LB - Rule 3}
                \coordinate (top) at (rel axis cs:0,1);
            \nextgroupplot[ymax = 3, xlabel = $\lambda$, axis x line = bottom, only marks]
                \addplot[black, mark = o] table[x = lambda, y = UB]{data_files/biobj2_UB_relax.dat};
                \addplot[black, mark = square] table[x = lambda, y = LB]{data_files/biobj2_LB_after_relax.dat};
                \addplot[magenta, mark = o] table[x = lambda, y = UB]{data_files/biobj2_UB_restart.dat};
                \addplot[red, mark = square] table[x = lambda, y = LB]{data_files/biobj2_LB_after_restart.dat};
                \addplot[blue, mark = square] table[x = lambda, y = LB]{data_files/biobj2_LB_before_cuts.dat};
                \addplot[green, mark = square] table[x = lambda, y = LB]{data_files/biobj2_LB_before_point.dat};
                \coordinate (bot) at (rel axis cs:1,0);
            \end{groupplot}
            \path (top-|current bounding box.west)--node[anchor=south,rotate=90] {Bounds} (bot-|current bounding box.west);
        \end{tikzpicture}
        \caption{Initial lower and upper bounds using different initialization rules for \eqref{eq:ti14}. Note the scaling of the vertical axis.}
        \label{fig:LB-UB_ti14}
    \end{subfigure}
    \caption{Results for solving \eqref{eq:ti14}.}
    \label{fig:plots_ti14}
\end{figure}

\begin{figure}[p]
    \centering
    \begin{subfigure}[t]{\textwidth}
        \centering
        \begin{tikzpicture}
            \begin{axis}[enlargelimits = false,
                        legend pos=north west,
                        ymin = -1,
                        width = 0.65\textwidth,
                        height = 0.5\textwidth,
                        xlabel = $\lambda$,
                        ylabel = Iterations,
                        legend pos=outer north east]
                \addplot[black, solid] table[x = lambda, y = relaxation]{data_files/oa_iters_slr.dat}; \addlegendentry{Relaxation}
                \addplot[red, dashed] table[x = lambda, y = rule1]{data_files/oa_iters_slr.dat}; \addlegendentry{Rule 1}
                \addplot[blue, dashdotted] table[x = lambda, y = rule2]{data_files/oa_iters_slr.dat}; \addlegendentry{Rule 2}
                \addplot[green, dashdotdotted] table[x = lambda, y = rule3]{data_files/oa_iters_slr.dat}; \addlegendentry{Rule 3}
            \end{axis}
        \end{tikzpicture}
        \caption{Number of iterations for the OA algorithm to converge for SLR by varying $\lambda$.}
        \label{fig:oa_iters_slr_lambda}
    \end{subfigure}
    \begin{subfigure}[c]{\textwidth}
        \centering
        \begin{tikzpicture}
            \begin{axis}[enlargelimits = false,
                        legend pos=north west,
                        ymin = 150,
                        width = 0.65\textwidth,
                        height = 0.5\textwidth,
                        xlabel = $\lambda$,
                        ylabel = Constraints,
                        legend pos=outer north east]
                \addplot[black, solid] table[x = lambda, y = relaxation]{data_files/nCons_slr.dat}; \addlegendentry{Relaxation}
                \addplot[black, densely dotted] table[x = lambda, y = relaxation]{data_files/ncons_presolved_slr_relaxation_lambda.dat}; \addlegendentry{Relaxation - Presolved}
                \addplot[red, loosely dashed] table[x = lambda, y = rule1]{data_files/nCons_slr.dat}; \addlegendentry{Rule 1}
                \addplot[red, densely dashed] table[x = iter, y = restart]{data_files/ncons_presolved_slr_restart_lambda.dat}; \addlegendentry{Rule 1 - Presolved}
                \addplot[blue, loosely dashdotted] table[x = lambda, y = rule2]{data_files/nCons_slr.dat}; \addlegendentry{Rule 2}
                \addplot[blue, densely dashdotted] table[x = iter, y = cuts]{data_files/ncons_presolved_slr_cuts_lambda.dat}; \addlegendentry{Rule 2 - Presolved}
                \addplot[green, loosely dashdotdotted] table[x = lambda, y = rule3]{data_files/nCons_slr.dat}; \addlegendentry{Rule 3}
                \addplot[green, densely dashdotdotted] table[x = iter, y = points]{data_files/ncons_presolved_slr_points_lambda.dat}; \addlegendentry{Rule 3 - Presolved}
            \end{axis}
        \end{tikzpicture}
        \caption{Number of constraints in the subproblem when the OA algorithm has converged for SLR by varying $\lambda$.}
        \label{fig:ncons_slr_lambda}
    \end{subfigure}
    \begin{subfigure}[b]{\textwidth}
        \centering
        \begin{tikzpicture}[every mark/.append style={mark size=1pt}]
            \begin{groupplot}[
            group style =   {group name = my plots,
                            group size = 1 by 2,
                            xticklabels at = edge bottom,
                            ylabels at = edge left,
                            yticklabels at = edge left,
                            vertical sep = 0pt},
                            width = 0.65\textwidth,
                            height = 0.3\textwidth,
                            enlargelimits = false]
            \nextgroupplot  [ymin = 325,
                            ymax = 410,
                            legend pos = outer north east,
                            axis x line = top,
                            only marks]
                \addplot[black, mark = o] table[x = lambda, y = UB]{data_files/slr_lambda_UB_relax.dat}; \addlegendentry{UB - Relaxation}
                \addplot[black, mark = square] table[x = lambda, y = LB]{data_files/slr_lambda_LB_after_relax.dat}; \addlegendentry{LB - Relaxation}
                \addplot[magenta, mark = o] table[x = lambda, y = UB]{data_files/slr_lambda_UB_restart.dat}; \addlegendentry{UB - Rule 1, 2 \& 3}
                \addplot[red, mark = square] table[x = lambda, y = LB]{data_files/slr_lambda_LB_after_restart.dat}; \addlegendentry{LB - Rule 1}
                \addplot[blue, mark = square] table[x = lambda, y = LB]{data_files/slr_lambda_LB_before_cuts.dat}; \addlegendentry{LB - Rule 2}
                \addplot[green, mark = square] table[x = lambda, y = LB]{data_files/slr_lambda_LB_before_point.dat}; \addlegendentry{LB - Rule 3}
                \coordinate (top) at (rel axis cs:0,1);
            \nextgroupplot[ymax = 325, xlabel = $\lambda$, axis x line = bottom, only marks]
                \addplot[black, mark = o] table[x = lambda, y = UB]{data_files/slr_lambda_UB_relax.dat};
                \addplot[black, mark = square] table[x = lambda, y = LB]{data_files/slr_lambda_LB_after_relax.dat};
                \addplot[magenta, mark = o] table[x = lambda, y = UB]{data_files/slr_lambda_UB_restart.dat};
                \addplot[red, mark = square] table[x = lambda, y = LB]{data_files/slr_lambda_LB_after_restart.dat};
                \addplot[blue, mark = square] table[x = lambda, y = LB]{data_files/slr_lambda_LB_before_cuts.dat};
                \addplot[green, mark = square] table[x = lambda, y = LB]{data_files/slr_lambda_LB_before_point.dat};
                \coordinate (bot) at (rel axis cs:1,0);
            \end{groupplot}
            \path (top-|current bounding box.west)--node[anchor=south,rotate=90] {Bounds} (bot-|current bounding box.west);
        \end{tikzpicture}
        \caption{Initial lower and upper bounds using different initialization rules for \eqref{eq:slr} by varying $\lambda$. Note the scaling of the vertical axis.}
        \label{fig:LB-UB_slr_lambda}
    \end{subfigure}
    \caption{Results for solving \eqref{eq:slr} by varying $\lambda$.}
    \label{fig:plots_slr_lambda}
\end{figure}

\begin{figure}[p]
    \centering
    \begin{subfigure}[t]{\textwidth}
        \centering
        \begin{tikzpicture}
            \begin{axis}[enlargelimits = false,
                        legend pos=north west,
                        ymin = -1,
                        width = 0.65\textwidth,
                        height = 0.5\textwidth,
                        xlabel = $\kappa$,
                        ylabel = Iterations,
                        legend pos=outer north east]
                \addplot[black, solid] table[x = iter, y = relaxation]{data_files/oa_iters_slr_kappa.dat}; \addlegendentry{Relaxation}
                \addplot[red, dashed] table[x = iter, y = rule1]{data_files/oa_iters_slr_kappa.dat}; \addlegendentry{Rule 1}
                \addplot[blue, dashdotted] table[x = iter, y = rule2]{data_files/oa_iters_slr_kappa.dat}; \addlegendentry{Rule 2}
                \addplot[green, dashdotdotted] table[x = iter, y = rule3]{data_files/oa_iters_slr_kappa.dat}; \addlegendentry{Rule 3}
            \end{axis}
        \end{tikzpicture}
        \caption{Number of iterations for the OA algorithm to converge for SLR by varying $\kappa$.}
        \label{fig:oa_iters_slr_kappa}
    \end{subfigure}
    \begin{subfigure}[c]{\textwidth}
        \centering
        \begin{tikzpicture}
            \begin{axis}[enlargelimits = false,
                        legend pos=north west,
                        ymin = 0,
                        width = 0.65\textwidth,
                        height = 0.5\textwidth,
                        xlabel = $\kappa$,
                        ylabel = Constraints,
                        xmin = 1,
                        xmax = 8,
                        legend pos=outer north east]
                \addplot[black, solid] table[x = iter, y = relaxation]{data_files/nCons_slr_kappa.dat}; \addlegendentry{Relaxation}
                \addplot[black, densely dotted] table[x = kappa, y = relaxation]{data_files/ncons_presolved_slr_kappa_relaxation_lambda.dat}; \addlegendentry{Relaxation - Presolved}
                \addplot[red, loosely dashed] table[x = iter, y = rule1]{data_files/nCons_slr_kappa.dat}; \addlegendentry{Rule 1}
                \addplot[red, densely dashed] table[x = iter, y = restart]{data_files/ncons_presolved_slr_kappa_restart.dat}; \addlegendentry{Rule 1 - Presolved}
                \addplot[blue, loosely dashdotted] table[x = iter, y = rule2]{data_files/nCons_slr_kappa.dat}; \addlegendentry{Rule 2}
                \addplot[blue, densely dashdotted] table[x = iter, y = cuts]{data_files/ncons_presolved_slr_kappa_cuts.dat}; \addlegendentry{Rule 2 - Presolved}
                \addplot[green, loosely dashdotdotted] table[x = iter, y = rule3]{data_files/nCons_slr_kappa.dat}; \addlegendentry{Rule 3}
                \addplot[green, densely dashdotdotted] table[x = iter, y = points]{data_files/ncons_presolved_slr_kappa_points.dat}; \addlegendentry{Rule 3 - Presolved}
            \end{axis}
        \end{tikzpicture}
        \caption{Number of constraints in the subproblem when the OA algorithm has converged for SLR by varying $\kappa$.}
        \label{fig:ncons_slr_kappa}
    \end{subfigure}
    \begin{subfigure}[b]{\textwidth}
    \centering
    \begin{tikzpicture}[every mark/.append style={mark size=1pt}]
        \begin{axis}[enlargelimits = false,
                    legend pos = outer north east,
                    width = 0.65\textwidth,
                    height = 0.5\textwidth,
                    xlabel = $\kappa$,
                    ylabel = Bounds,
                    only marks,
                    ymax = 450]
                \addplot[black, mark = o] table[x = kappa, y = UB]{data_files/slr_kappa_UB_relax.dat}; \addlegendentry{UB - Relaxation}
                \addplot[black, mark = square] table[x = kappa, y = LB]{data_files/slr_kappa_LB_after_relax.dat}; \addlegendentry{LB - Relaxation}
                \addplot[magenta, mark = o] table[x = kappa, y = UB]{data_files/slr_kappa_UB_restart.dat}; \addlegendentry{UB - Rule 1, 2 \& 3}
                \addplot[red, mark = square] table[x = kappa, y = LB]{data_files/slr_kappa_LB_after_restart.dat}; \addlegendentry{LB - Rule 1}
                \addplot[blue, mark = square] table[x = kappa, y = LB]{data_files/slr_kappa_LB_before_cuts.dat}; \addlegendentry{LB - Rule 2}
                \addplot[green, mark = square] table[x = kappa, y = LB]{data_files/slr_kappa_LB_before_point.dat}; \addlegendentry{LB - Rule 3}
        \end{axis}
    \end{tikzpicture}
    \caption{Initial lower and upper bounds using different initialization rules for \eqref{eq:slr} by varying $\kappa$.}
    \label{fig:LB-UB_slr_kappa}
    \end{subfigure}
    \caption{Results for solving \eqref{eq:slr} by varying $\kappa$.}
    \label{fig:plots_slr_kappa}
\end{figure}

\begin{figure}[p]
    \centering
    \begin{subfigure}[t]{\textwidth}
        \centering
        \begin{tikzpicture}
            \begin{axis}[enlargelimits = false,
                        legend pos=north west,
                        ymin = -1,
                        width = 0.65\textwidth,
                        height = 0.5\textwidth,
                        xlabel = Simulation,
                        xtick = {1, 2, 3, 4, 5, 6, 7, 8, 9, 10, 11, 12, 13, 14, 15},
                        ylabel = Iterations,
                        legend pos=outer north east]
                \addplot[black, solid] table[x = iter, y = relax]{data_files/oa_iters_mpc.dat}; \addlegendentry{Relaxation}
                \addplot[red, dashed] table[x = iter, y = restart]{data_files/oa_iters_mpc.dat}; \addlegendentry{Rule 1}
                \addplot[blue, dashdotted] table[x = iter, y = cuts]{data_files/oa_iters_mpc.dat}; \addlegendentry{Rule 2}
                \addplot[green, dashdotdotted] table[x = iter, y = points]{data_files/oa_iters_mpc.dat}; \addlegendentry{Rule 3}
            \end{axis}
        \end{tikzpicture}
        \caption{Number of iterations for the OA algorithm to converge for MPC.}
        \label{fig:oa_iters_mpc}
    \end{subfigure}
    \hfill
    \begin{subfigure}[c]{\textwidth}
        \centering
        \begin{tikzpicture}
            \begin{axis}[enlargelimits = false,
                        legend pos=north west,
                        ymin = 0,
                        width = 0.65\textwidth,
                        height = 0.5\textwidth,
                        xtick = {1, 2, 3, 4, 5, 6, 7, 8, 9, 10, 11, 12, 13, 14, 15},
                        xlabel = Simulation,
                        ylabel = Constraints,
                        legend pos=outer north east]
                \addplot[black, solid] table[x = iter, y = relax]{data_files/nCons_mpc.dat}; \addlegendentry{Relaxation}
                \addplot[black, densely dotted] table[x = iter, y = relaxation]{data_files/ncons_presolved_mpc_relaxation.dat}; \addlegendentry{Relaxation - Presolved}
                \addplot[red, loosely dashed] table[x = iter, y = restart]{data_files/nCons_mpc.dat}; \addlegendentry{Rule 1}
                \addplot[red, densely dashed] table[x = iter, y = restart]{data_files/ncons_presolved_mpc_restart.dat}; \addlegendentry{Rule 1 - Presolved}
                \addplot[blue, loosely dashdotted] table[x = iter, y = cuts]{data_files/nCons_mpc.dat}; \addlegendentry{Rule 2}
                \addplot[blue, densely dashdotted] table[x = iter, y = cuts]{data_files/ncons_presolved_mpc_cuts.dat}; \addlegendentry{Rule 2 - Presolved}
                \addplot[green, loosely dashdotdotted] table[x = iter, y = points]{data_files/nCons_mpc.dat}; \addlegendentry{Rule 3}
                \addplot[green, densely dashdotdotted] table[x = iter, y = points]{data_files/ncons_presolved_mpc_points.dat}; \addlegendentry{Rule 3 - Presolved}
            \end{axis}
        \end{tikzpicture}
        \caption{Number of constraints in the subproblem when the OA algorithm has converged for MPC.}
        \label{fig:ncons_mpc}
    \end{subfigure}
    \begin{subfigure}[b]{\textwidth}
    \centering
    \begin{tikzpicture}[every mark/.append style={mark size=1pt}]
        \begin{axis}[enlargelimits = false,
                    legend pos = outer north east,
                    width = 0.65\textwidth,
                    height = 0.5\textwidth,
                    xlabel = Simulation,
                    ylabel = Bounds,
                    only marks]
                \addplot[black, mark = o] table[x = iteration, y = UB]{data_files/mpc_UB_relax.dat}; \addlegendentry{UB - Relaxation}
                \addplot[black, mark = square] table[x = iteration, y = LB]{data_files/mpc_LB_after_relax.dat}; \addlegendentry{LB - Relaxation}
                \addplot[magenta, mark = o] table[x = iteration, y = UB]{data_files/mpc_UB_point.dat}; \addlegendentry{UB - Rule 1, 2 \& 3}
                \addplot[red, mark = square] table[x = iteration, y = LB]{data_files/mpc_LB_after_restart.dat}; \addlegendentry{LB - Rule 1}
                \addplot[blue, mark = square] table[x = iteration, y = LB]{data_files/mpc_LB_before_cuts.dat}; \addlegendentry{LB - Rule 2}
                \addplot[green, mark = square] table[x = iteration, y = LB]{data_files/mpc_LB_before_point.dat}; \addlegendentry{LB - Rule 3}
        \end{axis}
    \end{tikzpicture}
    \caption{Initial lower and upper bounds using different initialization rules for \eqref{eq:mpc}.}
    \label{fig:LB-UB_mpc}
    \end{subfigure}
    \caption{Results for solving \eqref{eq:mpc}.}
    \label{fig:plots_mpc}
\end{figure}

We now consider more detailed results of the algorithms in Figure \ref{fig:plots_ti4}, \ref{fig:plots_ti14}, \ref{fig:plots_slr_lambda}, \ref{fig:plots_slr_kappa} and \ref{fig:plots_mpc}. For each problem, there are three subplots. Plot (a) shows the number of iterations that the OA algorithm performs for each parameter value. Plot (b) shows the number of constraints in the MILP subproblem when the OA algorithm has converged for each parameter value. Plot (c) shows the initial lower and upper bounds of each rule.

We first consider the number of iterations. We see that relaxation and rule 1 have similar performance overall. For \eqref{eq:ti4} and \eqref{eq:ti14} relaxation performs slightly better, while the reverse happens for \eqref{eq:slr}. In all cases, rules 2 and 3 need fewer iterations than relaxation. This is a promising result since it means that applying the rules does indeed warm-start the algorithm. For rules 2 and 3, the number of iterations needed is similar. Interestingly it can be noted that the spikes, most noticeably visible in Figure \ref{fig:oa_iters_ti4} and \ref{fig:oa_iters_ti14} correspond to values of $\lambda$ where the optimal integer point changes. This means that as long as the same integer point is optimal for the MINLP, the warm-starting is highly efficient. We even observe that for many parameter values, only one iteration of the OA algorithm is needed to converge. This confirms the theory and intuition that were developed in Section \ref{sec:theory-properties}.

We can now also compare Figure \ref{fig:oa_iters_slr_lambda} and \ref{fig:oa_iters_slr_kappa}. This corresponds to the problem of varying either $\lambda$ or $\kappa$ in \eqref{eq:slr}. In the case of varying $\lambda$, rules 2 and 3 have a significant impact on the number of iterations needed. They require only 1 to 3 iterations to converge, compared to 150 to 200 iterations for the relaxation or rule 1. We compare this to the case of varying $\kappa$. For this case some improvement is observed from using rule 2 or rule 3, but it is less pronounced. The number of iterations needed to converge remains in the same magnitude. The difference between the two cases is that the integer points of interest remain similar in the case of varying $\lambda$ which is not true when varying $\kappa$. That is, the integer part of the optimal solution remains constant for several values of $\lambda$, while it almost always changes between different values of $\kappa$. Remember that changing $\kappa$ means changing the number of variables forced to zero which in turn is modelled with binary variables meaning that the number of binary variables forced to be zero changes. This confirms the intuition that the methods perform well when the integer part of the optimal solution remains constant for several parameter values.

Lastly the results for \eqref{eq:mpc}, plotted in Figure \ref{fig:oa_iters_mpc}, are more complicated. Some of the problems are for some reason harder to solve, especially the fifth problem. We also see that from the eighth simulation, rules 2 and 3 outperform the relaxation and rule 1. One possible explanation is that after some iterations, enough information has been collected for fast convergence. Another possible explanation is that the MPC control has stabilized meaning that few iterations are needed since it only needs to certify the optimality of the current control input. Until simulation six, rule 3 needs fewer iterations, meaning that it seems to adapt better to the differences between problems. On the other hand, from simulation seven rule 2 needs fewer iterations suggesting that when the control has converged reusing cuts is enough to quickly certify optimality.

We now investigate the number of constraints in the MILP subproblem when the OA algorithm has converged for each value of $\lambda$ in Figure \ref{fig:ncons_ti4}, \ref{fig:ncons:ti14}, \ref{fig:ncons_slr_lambda}, \ref{fig:ncons_slr_kappa} and \ref{fig:ncons_mpc}. The presolved lines correspond to the size of the problem when Gurobis presolve has been applied to the MILP subproblem that was constructed when the OA algorithm converged. For the problems \eqref{eq:ti4} and \eqref{eq:ti14} the trend is similar. We see that for relaxation, rule 1 and rule 3, the size of the problems remains fairly similar, while the size of the problems for rule 2 grows noticeably. However, when presolve is applied the size is of the same magnitude as for the other rules. This suggests that it would be possible to introduce some sort of cut selection to make sure that unnecessary cuts are not kept throughout the algorithm. Remember that rule 2 was the fastest of the methods, so one could say that the rule is saved by the highly efficient presolve implemented in commercial MIP solvers.

For the SLR problem, plotted in Figure \ref{fig:ncons_slr_lambda} and \ref{fig:ncons_slr_kappa}, the difference in the size of the subproblems is not as pronounced. We also note that presolve has less of an effect on the size of the problem. This suggests that there is less possibility of removing cuts during the algorithm without affecting the performance. The number of constraints for MPC, plotted in Figure \ref{fig:ncons_mpc}, has a similar evolution as for the SLR problem, with the exception that rule 3 keeps growing.

Finally, we analyze how the initial lower and upper bounds behave for each rule, which are plotted in Figure \ref{fig:LB-UB_ti4}, \ref{fig:LB-UB_ti14}, \ref{fig:LB-UB_slr_lambda}, \ref{fig:LB-UB_slr_kappa} and \ref{fig:LB-UB_mpc}. Note that an initially good lower or upper bound does not necessarily imply fast convergence. For \eqref{eq:ti4} and \eqref{eq:ti14}, the lower bounds for rule 1 are much worse than the other rules with rules 2 and 3 being slightly better than relaxation. For \eqref{eq:slr}, when varying $\lambda$, the lower bounds are similar, but the upper bounds for rules 1, 2 and 3 are better than the upper bounds for relaxation. For \eqref{eq:slr}, when varying $\kappa$, we instead observe that the lower bounds for relaxation are better than the warm-starting rules. In the case of \eqref{eq:mpc}, the lower bounds are similar for the rules but for the upper bounds relaxation is noticeably worse than rules 1, 2 and 3. In summary, the rule which achieves the best initial lower and upper bounds differs, but it is clear that the warm-starting rules perform better in several cases. An interesting observation is how poor the lower bound from rule 1 is in several cases. This once again underlines that the idea of warm-starting the OA algorithm only using a near-optimal solution is not an effective method.

\subsubsection{Summary of results}\label{sec:results_summary}
The results show that only using the optimal solution is not enough to achieve an effective warm-starting procedure and can even be worse than the standard technique of relaxation since the baseline and rule 1 perform similarly across the problems. This confirms the need for warm-starting procedures. We can also conclude that the results for rule 2 and 3 show that rather simple techniques can improve the performance when using the OA algorithm in this setting. In most cases we need to solve less than half as many MILP subproblems compared to relaxation when using rule 2 or 3.

Depending on the structure of the problem, the effectiveness varies with the most performance gain observed when the integer part of the optimal solution remains constant for several parameter values. On the downside, the size of the MILP subproblem can grow, most noticeably when using rule 2 since it is, in some cases, collecting many cuts that are no longer useful. The comparison with the size of the presolved problem shows that in some cases, most noticeably for the biobjective problems, this can potentially be handled by cut selection.

\section{Conclusions and future research}\label{sec:conclusions}
The purpose of this paper has been to investigate the potential of warm-starting techniques for solving sequences of convex MINLPs. We have chosen to focus on the OA algorithm to solve each convex MINLP. An example illustrated why the trivial warm-starting approach of starting in the optimal solution is not sufficient. After that, theoretical results were developed which showed that under some conditions warm-starting can be used efficiently resulting in convergence in one iteration. Practical warm-starting rules were then presented that can be integrated into an algorithm to solve a parameterized convex MINLP. The main contributions were the two rules, called cut-tightening and point-based, that were compared to the baseline rules of using the continuous relaxation or starting in the optimal solution to the last parameter value. Using five examples, numerical results were produced for the different rules. The results showed that there is significant potential in warm-starting techniques, since even with simple rules large gains can be achieved. It was noted that especially in the case when the integer part of the optimal solution remains constant for several parameter values, the warm-starting is highly efficient, while it struggles more when the integer solution switches between parameter values. By comparing the size of the MILP subproblem to the size after presolve, it was suggested that there is potential for cut selection to reduce the size of the subproblems. The conclusion is still that warm-starting improves the efficiency of solving parameterized convex MINLPs and has important potential for further research.

We can identify several interesting directions for future research based on this paper. First, an interesting question is how warm-starting can be used in other cutting plane algorithms for solving convex MINLPs. The conclusion from this paper that warm-starting can achieve a performance improvement should also hold for other algorithms. Second, implementing and integrating the techniques in established software packages, for instance SCIP, and further refining the methods would make the methods accessible and enable more extensive numerical tests. Lastly, adding some cut selection to the cut based method to handle the growth of the MILP subproblems could be beneficial for performance.

\bibliographystyle{amsplain}
\bibliography{bibliography.bib}

\providecommand{\bysame}{\leavevmode\hbox to3em{\hrulefill}\thinspace}
\providecommand{\MR}{\relax\ifhmode\unskip\space\fi MR }
\providecommand{\MRhref}[2]{%
  \href{http://www.ams.org/mathscinet-getitem?mr=#1}{#2}
}
\providecommand{\href}[2]{#2}
\begin{thebibliography}{10}

\bibitem{abhishek2010filmint}
K~Abhishek, S~Leyffer, and J~Linderoth, \emph{Filmint: An outer approximation-based solver for convex mixed-integer nonlinear programs}, {INFORMS} {J}ournal on {C}omputing \textbf{22} (2010), no.~4, 555--567.

\bibitem{Acevedo_Salgueiro_2003}
J~Acevedo and M~Salgueiro, \emph{An efficient algorithm for convex multiparametric nonlinear programming problems}, {I}ndustrial \& {E}ngineering {C}hemistry {R}esearch \textbf{42} (2003), no.~23, 5883--5890.

\bibitem{belotti_2013}
P~Belotti, C~Kirches, S~Leyffer, J~Linderoth, J~Luedtke, and A~Mahajan, \emph{Mixed-integer nonlinear optimization}, {A}cta {N}umerica \textbf{22} (2013), 1--131.

\bibitem{benson2011mixed}
HY~Benson, \emph{Mixed integer nonlinear programming using interior-point methods}, Optimization {M}ethods and {S}oftware \textbf{26} (2011), no.~6, 911--931.

\bibitem{Bertsimas_Parys_2020}
D~Bertsimas and B~Van~Parys, \emph{Sparse high-dimensional regression: Exact scalable algorithms and phase transitions}, {T}he {A}nnals of {S}tatistics \textbf{48} (2020), no.~1, 300--323.

\bibitem{Julia-2017}
J~Bezanson, A~Edelman, S~Karpinski, and VB~Shah, \emph{Julia: {A} fresh approach to numerical computing}, {SIAM} {R}eview \textbf{59} (2017), no.~1, 65--98.

\bibitem{Bonami_2008}
P~Bonami, LT~Biegler, AR~Conn, G~Cornu{\'e}jols, IE~Grossmann, CD~Laird, J~Lee, A~Lodi, F~Margot, N~Sawaya, and A~W{\"a}chter, \emph{An algorithmic framework for convex mixed integer nonlinear programs}, {D}iscrete {O}ptimization \textbf{5} (2008), no.~2, 186--204.

\bibitem{bonami_2012}
P~Bonami, M~Kilin{\c{c}}, and J~Linderoth, \emph{Algorithms and software for convex mixed integer nonlinear programs}, Mixed Integer Nonlinear Programming (New York) (J~Lee and S~Leyffer, eds.), {T}he {IMA} {V}olumes in {M}athematics and its {A}pplications, vol. 154, Springer, 2012, pp.~1--39.

\bibitem{boukouvala_2016}
F~Boukouvala, R~Misener, and CA~Floudas, \emph{Global optimization advances in mixed-integer nonlinear programming, {MINLP}, and constrained derivative-free optimization, {CDFO}}, {E}uropean {J}ournal of {O}perational {R}esearch \textbf{252} (2016), no.~3, 701--727.

\bibitem{Boyd_2004}
SP~Boyd and L~Vandenberghe, \emph{Convex {O}ptimization}, {C}ambridge {U}niversity {P}ress, Cambridge, 2004.

\bibitem{Conforti_2014}
M~Conforti, G~Cornu{\'e}jols, and G~Zambelli, \emph{Integer programming}, Springer, Cham, 2014.

\bibitem{wine_data}
P~Cortez, A~Cerdeira, F~Almeida, T~Matos, and J~Reis, \emph{{Wine Quality}}, {UCI} {M}achine {L}earning {R}epository, 2009, \url{https://doi.org/10.24432/C56S3T}.

\bibitem{Dakin_1965}
RJ~Dakin, \emph{A tree-search algorithm for mixed integer programming problems}, {C}omputer {J}ournal \textbf{8} (1965), no.~3, 250--255.

\bibitem{dua2000algorithm}
V~Dua and EN~Pistikopoulos, \emph{An algorithm for the solution of multiparametric mixed integer linear programming problems}, {A}nnals of {O}perations {R}esearch \textbf{99} (2000), 123--139.

\bibitem{Dua_2009}
\bysame, \emph{Parametric mixed integer nonlinear optimization}, {E}ncyclopedia of {O}ptimization (CA~Floudas and PM~Pardalos, eds.), Springer, Boston, 2009, pp.~2920--2924.

\bibitem{Duran_Grossmann_1986}
MA~Duran and IE~Grossmann, \emph{An outer-approximation algorithm for a class of mixed-integer nonlinear programs}, {M}athematical {P}rogramming \textbf{36} (1986), no.~3, 307--339.

\bibitem{ambrosio_2013}
C~D’Ambrosio and A~Lodi, \emph{Mixed integer nonlinear programming tools: an updated practical overview}, {A}nnals of {O}perations {R}esearch \textbf{204} (2013), 301--320.

\bibitem{Ehrgott_2005}
M~Ehrgott, \emph{Multicriteria {O}ptimization}, 2 ed., Springer, Berlin, Heidelberg, 2005.

\bibitem{Eichfelder_2021}
G~Eichfelder, \emph{Twenty years of continuous multiobjective optimization in the twenty-first century}, EURO Journal on Computational Optimization \textbf{9} (2021), 100014.

\bibitem{Biobj_test_probs}
G~Eichfelder, T~Gerlach, and L~Warnow, \emph{Test instances for multiobjective mixed-integer nonlinear optimization}, {G}eometry and {N}on-{C}onvex {O}ptimization (PM~Pardalos and TM~Rassias, eds.), {S}pringer {O}ptimization and {I}ts {A}pplications, vol. 223, Springer, 2025, pp.~71--99.

\bibitem{Fallah_2024}
S~Fallah, TK~Ralphs, and NL~Boland, \emph{On the relationship between the value function and the efficient frontier of a mixed integer linear optimization problem}, 2024.

\bibitem{fiacco_2008}
AV~Fiacco, \emph{Sensitivity and stability in {NLP}: {C}ontinuity and differential stability}, {E}ncyclopedia of {O}ptimization (Christodoulos~A. Floudas and Panos~M. Pardalos, eds.), Springer, Boston, 2008, pp.~3467--3471.

\bibitem{fletcher_1994}
R~Fletcher and S~Leyffer, \emph{Solving mixed integer nonlinear programs by outer approximation}, Mathematical Programming \textbf{66} (1994), no.~3, 327--349.

\bibitem{floudas_1995}
CA~Floudas, \emph{Nonlinear and {M}ixed-integer {O}ptimization: {F}undamentals and {A}pplications}, Topics in {C}hemical {E}ngineering, {O}xford {U}niversity {P}ress, {N}ew {Y}ork, 1995.

\bibitem{Gal_1995}
T~Gal, \emph{Postoptimal analyses, parametric programming, and related topics : degeneracy, multicriteria decision making, redundancy}, 2 ed., W. de Gruyter, Berlin, 1995.

\bibitem{Geoffrion_1972}
AM~Geoffrion, \emph{Generalized benders decomposition}, {J}ournal of {O}ptimization {T}heory and {A}pplications \textbf{10} (1972), no.~4, 237--260.

\bibitem{Gupta_1986}
OK~Gupta and AR~Ravindran, \emph{Branch and bound experiments in convex nonlinear integer programming}, {M}anagement {S}cience \textbf{31} (1985), no.~12, 1533--1546.

\bibitem{Kronqvist_2019}
J~Kronqvist, DE~Bernal, A~Lundell, and IE~Grossmann, \emph{A review and comparison of solvers for convex {MINLP}}, {O}ptimization and {E}ngineering \textbf{20} (2019), no.~2, 397--455.

\bibitem{Land_Doig_1960}
AH~Land and AG~Doig, \emph{An automatic method of solving discrete programming problems}, Econometrica \textbf{28} (1960), no.~3, 497--520.

\bibitem{Lubin2023}
M~Lubin, O~Dowson, J~{Dias Garcia}, J~Huchette, B~Legat, and {JP} Vielma, \emph{{JuMP} 1.0: {R}ecent improvements to a modeling language for mathematical optimization}, {M}athematical {P}rogramming {C}omputation (2023), 581--589.

\bibitem{Miettinen_1999}
K~Miettinen, \emph{Nonlinear multiobjective optimization}, International {S}eries in {O}perations {R}esearch \& {M}anagement {S}cience, vol.~12, Kluwer Academic Publishers, Boston, 1999.

\bibitem{mitsos2009parametric}
A~Mitsos and PI~Barton, \emph{Parametric mixed-integer 0--1 linear programming: the general case for a single parameter}, {E}uropean {J}ournal of {O}perational {R}esearch \textbf{194} (2009), no.~3, 663--686.

\bibitem{Nocedal_Wright_2006}
J~Nocedal and SJ~Wright, \emph{{N}umerical {O}ptimization}, 2 ed., {S}pringer {S}eries in {O}perations {R}esearch and {F}inancial {E}ngineering, Springer, Berlin, 2006.

\bibitem{Papalexandri_Dimkou_1998}
KP~Papalexandri and TI~Dimkou, \emph{A parametric mixed-integer optimization algorithm for multiobjective engineering problems involving discrete decisions}, {I}ndustrial \& {E}ngineering {C}hemistry {R}esearch \textbf{37} (1998), no.~5, 1866--1882.

\bibitem{Rudin_1976}
W~Rudin, \emph{Principles of mathematical analysis}, 3 ed., McGraw-Hill, New York, 1976.

\bibitem{mpc}
B~Stellato and G~Banjac, \emph{Model predictive control ({MPC})}, 2021, \url{https://osqp.org/docs/examples/mpc.html#}.

\bibitem{trespalacios_2014}
F~Trespalacios and IE~Grossmann, \emph{Review of mixed-integer nonlinear and generalized disjunctive programming methods}, {C}hemie {I}ngenieur {T}echnik \textbf{86} (2014), no.~7, 991--1012.

\bibitem{Westerlund_1995}
T~Westerlund and F~Pettersson, \emph{An extended cutting plane method for solving convex {MINLP} problems}, {C}omputers \& {C}hemical {E}ngineering \textbf{19} (1995), no.~1, 131--136.

\end{thebibliography}

\end{document}